\documentclass[12pt]{amsart}
\title{Supercompact Extender Based Magidor-Radin Forcing}
\author{Carmi Merimovich}
\address{ 
School of Computer Science \\
Tel-Aviv Academic College \\
Rabenu Yeroham St.\\
Tel-Aviv 68182 \\
Israel
}
\email{carmi@cs.mta.ac.il}
\date{August 1, 2016}


\usepackage{enumitem}
\usepackage{amssymb}
\usepackage{doi}
\usepackage{undertilde}

\usepackage{ifsym}
\usepackage{accents}
\newcommand{\vbm}{\accentset{\mid}{\gm}}
\newcommand{\vbn}{\accentset{\mid}{\gn}}

\def\mypic{0}		

\ifnum\mypic=1
\usepackage{tikz}
\fi

\theoremstyle{plain}
\newtheorem{theorem}{Theorem}[section]
\newtheorem*{theorem*}{Theorem}
\newtheorem*{Expected*}{Aim}
\newtheorem{lemma}[theorem]{Lemma}

\newtheorem{claim}[theorem]{Claim}
\newtheorem{corollary}[theorem]{Corollary}

\theoremstyle{definition}
\newtheorem{definition}[theorem]{Definition}

\theoremstyle{remark}

\numberwithin{equation}{theorem}
\numberwithin{lemma2}{theorem}

\usepackage{cleveref}
\creflabelformat{section}{#2\protect\numberstringnum{#1}[n]#3}
\crefname{lemma}{lemma}{lemmas}
\crefname{clause}{clause}{clauses}
\crefname{claim}{claim}{claims}

\DeclareMathOperator{\OB}{OB}

\newcommand{\Es}{\Bar{E}}

\newcommand{\gns}{\Bar{\gn}}
\newcommand{\gts}{\Bar{\gt}}
\newcommand{\gnv}{{\Vec{\gn}}}
\newcommand{\gms}{\Bar{\gm}}
\newcommand{\gmv}{{\Vec{\gm}}}

\newcommand{\ga}{\alpha}
\newcommand{\gb}{\beta}
\newcommand{\gc}{\chi}

\newcommand{\gD}{\Delta}
\newcommand{\gee}{\epsilon}

\newcommand{\gi}{\iota}

\newcommand{\gk}{\kappa}
\newcommand{\gl}{\lambda}
\newcommand{\gm}{\mu}
\newcommand{\gn}{\nu}

\newcommand{\gp}{\pi}

\newcommand{\gr}{\rho}
\newcommand{\gs}{\sigma}

\newcommand{\gt}{\tau}

\newcommand{\gw}{\omega}
\newcommand{\gx}{\xi}
\newcommand{\gz}{\zeta}

\newcommand{\leftexp}[2]{{\vphantom{#2}}^{#1}{#2}}
\newcommand{\func}{\mathrel{:}}
\newcommand{\VN}[1]{\Check{#1}}
\newcommand{\CN}[1]{\underset{\widetilde{}}{#1}}
\newcommand{\GN}[1]{\utilde{#1}}

\newcommand{\union}{\cup}
\newcommand{\bigunion}{\bigcup}
\newcommand{\intersect}{\cap}
\newcommand{\bigintersect}{\bigcap}
\newcommand{\forces}{\mathrel\Vdash}

\newcommand{\incompatible}{\perp}
\newcommand{\ncompatible}{\perp}
\newcommand{\compatible}{\parallel}
\newcommand{\decides}{\mathrel\Vert}

\newcommand{\subelem}{\prec}

\newcommand{\append}{\mathop{{}^\frown}}

\newcommand{\restricted}{\mathrel{\restriction}}
\newcommand{\power}[1]{\lvert#1\rvert}
\DeclareMathOperator{\PSet}{{\mathcal{P}}}
\newcommand{\On}{\ensuremath{\text{On}}}

\newcommand{\ordered}[1]{\ensuremath{\langle #1 \rangle}}
\newcommand{\set}[1]{\ensuremath{\{ #1 \}}}
\newcommand{\setof}[2]{\ensuremath{\{ #1 \mid #2 \}}}
\newcommand{\ordof}[2]{\ensuremath{\ordered{ #1 \mid #2 }}}
\newcommand{\formula}[1]{{}\text{``} #1 {}\text{''}}

\DeclareMathOperator{\mo}{o}
\DeclareMathOperator{\otp}{otp}
\DeclareMathOperator{\crit}{crit}
\DeclareMathOperator{\dom}{dom}

\DeclareMathOperator{\cf}{cf}

\DeclareMathOperator{\Ult}{Ult}
\DeclareMathOperator{\Suc}{Suc}
\DeclareMathOperator{\Lev}{Lev}

\DeclareMathOperator{\ran}{ran}
\DeclareMathOperator*{\dintersect}{\triangle}

\DeclareMathOperator{\mc}{mc}

\newcommand{\PP}{\mathbb{P}}

\DeclareMathOperator{\ES}{ES}
\newcommand{\vE}{\Vec{E}}

\newcommand{\bT}{\leftexp{<\gw}{T}}

\newcommand{\ogn}{\mathring{\gn}}
\newcommand{\ogm}{\mathring{\gm}}
\makeatletter
\g@addto@macro\bfseries{\boldmath}
\makeatother
\begin{document}
\maketitle
\begin{abstract}
The extender based Magidor-Radin forcing is being generalized
	to supercompact type extenders.
\end{abstract}
\section{Introduction}
This work\footnote{
			Most of this work  was done somewhat
						after	\cite{Merimovich2011c} was completed.
				Lacking an application, which admittedly was lacking 
						also in \cite{Merimovich2011c},
	 							it was mainly distributed among interested parties.
				Gitik, observing the utility of this forcing to some HOD constructions
						(see \cite{GitikMerimovichPreprint}),
			has urged us to  bring the work into publishable state.}
 continues the project of of generalizing the extender based Prikry forcing
 		 \cite{GitikMagidor1992} to larger and larger cardinals.
In \cite{Merimovich2003,Merimovich2011b} the methods introduced in
		 \cite{GitikMagidor1992}
	(which generalized Prikry forcing \cite{Prikry1968} from using a measure to using an extender),
	were used to generalize the Magidor \cite{Magidor1978} and Radin \cite{Radin1982} forcing notions to use a sequence of extenders.
In a different direction \cite{Merimovich2011c} used the methods of \cite{GitikMagidor1992}
	to define the extender based Prikry forcing over extenders which have higher directedness properties than their critical point.
Such extenders give rise to supercompact type embeddings.
Generalization of Prikry forcing to fine ultrafilters yielding supercompact type embeddings
	appeared in \cite{Magidor1977I}.
Extending this forcing notion to Magidor-Radin type forcing notions were done
	in \cite{ForemanWoodin1991} and \cite{Krueger2007}.
In the current paper we use extenders with higher directedness properties to define
	the extender based Magidor-Radin forcing notion.
All of the forcing notions mentioned above are of course of Prikry type.
For more information on Prikry type forcing notions one should consult \cite{Gitik2010}.

Before stating the theorem of this paper we need to make some notions precise. 
Assume $E$ is an extender.
We let $j_E\func V \to M \simeq \Ult(V,E)$ be the natural embedding of $V$ into the transitive collpase of the
	ultrapower $\Ult(V,E)$.
We denote by $\crit E$ the critical point of the embedding $j_E$.
In  principle, an extender is a directed family of ultrafilters and projections.
We denote by $\gl(E)$ a degree of directedness holding for  the extender $E$.
We do not require $\gl(E)$ to be optimal, i.e., $\gl(E)$ is not necessarily the minimum cardinal
	for which $E$ is not $\gl(E)^+$-directed.
Note $M \supseteq \leftexp{<\gl(E)}{M}$.

A sequence of extenders $\Vec{E} = \ordof{E_\gx}{\gx < \mo(\Vec{E})}$, all with the same critical point $\crit E_\gx$
		and the same directedness size  $\gl(E_\gx)$,
		is said be Mitchell increasing
		if for each $\gx < \mo(\Vec{E})$ we have $\ordof{E_{\gx'}}{\gx' < \gx} \in M_\gx \simeq \Ult(V,E_\gx)$.
We will denote by $\crit(\Vec{E})$ and $\gl(\Vec{E})$  the common values of  $\crit E_\gx$ and $\gl(E_\gx)$, respectively.

If $\Vec{E} = \ordof{E_\gx}{\gx < \mo(\Vec{E})}$ is a Mithcell increasing sequence of extenders
	and $\ga \in [\crit \vE, j_{E_0}(\gk))$ then $\Es = \ordered{\ga, \Vec{E}}$ is said to be an extender sequence.
Hence an extender sequence is an ordered pair with the first coordinate being an ordinal and the
	second coordinate being a Mitchell increasing sequence of extenders.
Note that an empty sequence of extenders is legal in an extender sequence, e.g., $\ordered{\ga, \ordered{}}$
	is an extender sequence.
Let $\ES$ be the collection of extender sequences.
If $\Es$ is an extender sequence then we denote the projections to the first and second coordinates
		by $\mathring{\Es}$ and $\accentset{\mid}{\Es}$, respectively.
The ordinals at the first coordinate of an extender sequence induce an order $<$ on $\ES$
	by setting $\gns < \gms$ if $\mathring{\gns} < \mathring{\gms}$.
We lift the functions defined on the Mitchell increasing sequence of extenders to extender sequences in the obvious way, i.e., $\mo(\Es) = \mo(\accentset{\mid}{\Es})$ and $\gl(\Es)=\gl(\accentset{\mid}{\Es})$.
We will also abuse notation by writing $\Es_\gx$ for the extender $E_\gx$.
There are two restrictions we have on $\gl(\Vec{E})$.
The first one seems a bit technical.
We demand $\gl(\Es)^{<\crit\Es} = \gl(\Es)$ due to  limitations we encountered in \cref{GetGoodPair}.
The second one is more substantial.
We demand $\gl(\Es) \leq j_{E_0}(\crit(E_0))$.
(It seems this last demand can be removed for the special case $\mo(\Vec{E}) = 1$.)
With all these preliminaries at hand we can write the theorem proved in this paper.
\begin{theorem*}
Assume the GCH.
Let $\Vec{E}$ be a Mitchell increasing sequence 
		such that $\gl(\vE) < j_{E_0}(\crit(\Vec{E}))$ and 
		$\gm^{<\crit \vE} < \gl(\vE)$ for each $\gm < \gl(\vE)$.
Furthermore, assume 
			$\gee \leq j_{E_0}(\gk)$.
Then there is a forcing notion $\PP(\Vec{E}, \gee)$ such that the following hold in $V[G]$, where $G \subseteq \PP(\Vec{E}, \gee)$ is generic.
There is a set $G^\gk \subseteq \ES$ such that
	$G^\gk \union \set{ \ordered{\crit \Vec{E}, \Vec{E}} }$ is increasing and
	for each $\gns \in G^\gk \union \set{ \ordered{\crit \Vec{E}, \Vec{E}} }$ such that 
				$\mo(\gns) > 0$ the following hold:
\begin{enumerate}
\item
		$\setof{\crit \gms} {\gms \in G^\gk, \gms < \gns} \subseteq \mathring{\gns}$ is a club.
\item
		$\crit \gns$ and $\gl(\gns)$ are preserved in $V[G]$, and
		$(\crit \gns^+ = \gl(\gns))^{V[G]}$.
\item
		If $\mo(\gns)<\crit \gns$ is $V$-regular then $\cf \crit \gns = \cf \mo(\gns)$ in $V[G]$.
\item (Gitik)
		If $\mo(\gns) \in [\crit \gns, \gl(\gns))$ and $\cf(\mo(\gns)) \geq \crit(\gns)$
					 then $\cf \crit \gns = \gw$ in $V[G]$.
\item
		If $\mo(\gns) \in [\crit \gns, \gl(\gns))$ and $\cf(\mo(\gns)) < \crit(\gns)$
					 then $\cf \crit \gns = \cf\mo(\gns)$ in $V[G]$.		
\item
		If $\mo(\gns)= \crit(\gns)$ then $\cf \crit \gns  = \gw$ in $V[G]$.
\item
		If $\mo(\gns)= \gl(\gns)$ then $\crit \gns$ is regular in $V[G]$.
	\item
		If $\mo(\gns)= \gl(\gns)^{++}$  then $\crit \gns$ is measurable in $V[G]$.
\item
		$2^{\crit\gns} = \max \set{\gl(\gns), \power{\gee}}$.
\end{enumerate}
\end{theorem*}

Thus for example,
	if we assume $\ordof{E_\gx}{\gx < \gw_1}$ is a Mitchell increasing sequence
	of extenders on $\gk$ giving rise to a $<\gk^{++}$-closed elementary embeddings (and no more), 
		then in the generic extension $\gk$ will change its cofinality to $\gw_1$,
				and $\gk^+$ would be collapsed.
				Moreover, there is a club of ordertype $\gw_1$ cofinal in $\gk$,
						and for each limit point $\gt$ in this club $\gt^+$ of the ground model
								is collapsed. The GCH would be preserved, and no other cardinals are collapsed.
								
	As another example,
	assume $\ordof{E_\gx}{\gx < \gw_1}$ is a Mitchell increasing sequence
	of extenders on $\gk$ giving rise to a $<\gk^{++}$-close elementary embeddings
		which are also $\gk^{+3}-strong$ (and no more), 
		then in the generic extension $\gk$ will change its cofinality to $\gw_1$,
				and $\gk^+$ would be collapsed.
				Moreover, there is a club of ordertype $\gw_1$ cofinal in $\gk$,
						and for each limit point $\gt$ in this club $\gt^+$ of the ground model
								is collapsed. 
									In this case we get $2^\gk = \gk^{++}$ and $2^\gt = \gt^{++}$
											for the limit points of the club.
			In fact we have $2^\gk = (\gk^{+3})_V$ and $2^\gt = (\gt^{+3})_V$,
					and we see only gap-2 in the generic extension since $\gk^+$ of the ground mode
					gets collapsed
							as do all the $\gt^+$ of the ground model.
					No other cardinal get collapsed.

The structure of the work is as follows.
In section \ref{sec:ExtendersAndNormality} a formulation of extenders useful
	for $\gl$-directed extenders is presented, and an appropriate diagonal intersection
			operation is introduced.
In section \ref{sec:Forcing} the forcing notion is defined and the properties of it which do not rely on
	understanding the dense subsets of the forcing are presented.
In section \ref{sec:Dense} claims regarding the dense subsets of the forcing notion are presented.
	This section is highly combinatorial in nature.
In section \ref{sec:KappaProperties} the influence of
	$\mo(\vE)$ on the properties of
	$\gk$ in the generic extension is shown.
	The claims here  rely on the structure of the dense subsets as analyzed in section \ref{sec:Dense}.

This work is self contained assuming large cardinals and forcing are known.
\section{$\gl$-Directed Extenders and Normality} \label[section]{sec:ExtendersAndNormality}
Assume the GCH.
Let $\vE = \ordof {E_\gx}{\gx < \mo(\Vec{E})}$ be a Mitchell increasing sequence of  
			$\gl$-directed extenders 
	such that $\gl \leq j_{E_0}(\gk)$ is regular and $\gl^{<\gk} = \gl$,
		where $\gk = \crit \Vec{E}$.
		For each $\gx < \mo(\Vec{E})$ let
			$j_{E_\gx} \func V \to M_\gx \simeq \Ult(V, E_\gx)$ be the natural embedding.
Assume $d \in [\gee]^{<\gl}$ and $\power{d}+1 \subseteq d$.
We let $\OB(d)$ be the set of functions $\gn \func \dom \gn \to \ES$ such that
			$\gk \in \dom \gn \subseteq d$, 
							and if $\ga, \gb\in \dom \gn$ and $\ga < \gb$ then $\ogn(\ga) < \ogn(\gb)$.
Define an order on $\OB(d)$ by saying for each pair $\gn,\gm \in \OB(d)$
	that $\gn < \gm$ if
		$\dom \gn \subseteq \dom \gm$,
		$\power{\gn} < \ogm(\gk)$,
		and for each $\ga \in \dom \gn$,
			$\gn(\ga) < \ogm(\gk)$.

For $\gx < \mo(\vE)$ and a set $d \in [\gee]^{<\gl}$  define the measure $E_\gx(d)$ on $\OB(d)$ as follows:
		\begin{align*}
			X \in E_\gx(d) \iff  \setof {\ordered{j_{E_\gx}(\ga), \ordered{\ga, \ordof{E_{\gx'}}
							{\gx' < \gx}}}} {\ga \in d} \in j_{E_\gx}(X).
		\end{align*}
For a set  $d \in [\gee]^{<\gl}$
 let $\Vec{E}(d) = \bigintersect \setof{E_\gx(d}{\gx < \mo(\Es)}$.
It is clear  $E_\gx(d)$ is a $\gk$-complete ultrafilter over $\OB(d)$ and
	$\vE(d)$ is a $\gk$-complete filter over $\OB(d)$.
In addition to this, the filter $\vE(d)$ has a useful normality property
					 with a matching diagonal intersection soon to be introduced.
\begin{claim}
If $S \subseteq \OB(d)$, $\gn^* \in j_{E_\gx}(S)$, and $\gn^* < \mc_\gx(d)$,
	then there is $\gn \in S$ such that
			$\gn^* = j_{E_\gx}(\gn)$.
\end{claim}
Assume $S \subseteq \OB(d)$ and for each $\gn \in S$ there is a set $X(\gn) \subseteq \OB(d)$.
Define the diagonal intersection of the family $\setof{X(\gn)}{\gn \in S}$ as follows:
\begin{align*}
\dintersect_{\gn \in S}X(\gn) = \setof {\gn \in \OB(d)}{\forall \gm \in S\ (\gm < \gn \implies \gn \in X(\gm))}.
\end{align*}
\begin{lemma} \label[lemma]{DiagonalIsBig}
Assume  $S \subseteq \OB(d)$,
				 and for each $\gn \in S$, $X(\gn) \in \vE(d)$.
Then $X^* = \dintersect_{\gn \in S}X(\gn) \in \vE(d)$.
\end{lemma}
\begin{proof}
We need to show for each $\gx < \mo(\vE)$, $\mc_\gx(d) \in j_{E_\gx}(X^*)$.
I.e., we need to show $\mc_\gx(d) \in j_{E_\gx}(X)(\gn^*)$
		for each $\gn^* \in j_{E_\gx}(S)$ such that
		$\gn^* < \mc_\gx(d)$.
Fix $\gn^* \in j_{E_\gx}(S)$ such that $\gn^* < \mc_\gx(d)$.
There is $\gn \in S$ such that $\gn^* = j_{E_\gx}(\gn)$.
Hence
		$j_{E_\gx}(X)(\gn^*) = j_{E_\gx}(X(\gn))$.
Since $X(\gn) \in E_\gx(d)$ we get $\mc_\gz(d) \in j_{E_\gz}(X(\gn))$,
		by which we are done.
\end{proof}
The diagonal intersection above can be generalized to work with more than one measure in the following way.
A set $T \subseteq \leftexp{n}{\OB(d)}$, where $n<\gw$, is said to be a tree if the following hold:
\begin{enumerate}
\item
	Each $\ordered{\gn_0, \dotsc, \gn_{n-1}} \in T$ is increasing.
\item
		For each $k<n$ and $\ordered{\gn_0, \dotsc, \gn_{n-1}} \in T$ we have
				$\ordered{\gn_0, \dotsc, \gn_k} \in T$.
\end{enumerate}
	Assume $T \subseteq \leftexp{n}{\OB(d)}$ is a tree and $\ordered{\gn} \in T$.
	Set $T_{\ordered{\gn}}= \setof {\ordered{\gm_0, \dotsc, \gm_{n-2}}}
				{\ordered{\gn, \gm_0, \dotsc, \gm_{n-2}} \in T}$.
Denote the $k$-level of the tree $T$ by $\Lev_k(T)$, i.e.,
		$\Lev_k(T) = T \intersect \leftexp{k+1}{\OB(d)}$.
We will use $\gnv$ as a shorthand for $\ordered{\gn_0, \dotsc, \gn_{n-1}}$.
For each $\gnv \in T$ we define the successor level of $\gnv$ in $T$ by setting
	$\Suc_T(\gnv) = \setof {\gm}{\gnv \append \gm \in T}$.

A tree $S\subseteq \leftexp{n}{\OB(d)}$, with all maximal branches having the same finite height $n<\gw$,
		 is said to be an  $\vE(d)$-tree if the following hold:
\begin{enumerate}
	\item
	There is $\gx < \mo(\vE)$ such that
			$\Lev_0(S) \in E_\gx(d)$.
	\item
		For each $\gnv\append \gm \in S$ there is $\gx < \mo(\vE)$ such that
				$\Suc_S(\gnv) \in E_\gx(d)$.
\end{enumerate}
If $S$ is a tree of finite height $n<\gw$ then we write
	$\Lev_{\max} S$ for $\Lev_{n-1}S$.

Assume $S$ is an $\vE(d)$-tree, and for each $\gnv \in \Lev_{\max}(S)$
	there is a set $X(\gnv) \subseteq \OB(d)$.
By recursion define the diagonal intersection of the family
	$\setof {X(\gnv)}{\gnv \in \Lev_{\max}S}$ by setting
	$\dintersect \setof {X(\gnv)}{\gnv \in \Lev_{\max}S} = 
			\dintersect \setof {X^*(\gnv)}{\gnv \in \Lev_{\max}S^*}$,
	where
	   $S^* = S \intersect \leftexp{n-1}{\OB(d)}$
	   and
	   	$X^*(\gmv) = \dintersect \setof {X(\gmv \append \ordered{\gn})}{\gn \in S_{\ordered{\gmv}}}$.
The following is immediate.
\begin{corollary}	
Assume $S$ is an $\vE(d)$-tree, and for each $\gnv \in \Lev_{\max}(S)$
	there is a set $X(\gnv) \in \vE(d)$.
	Then $\dintersect \setof {X(\gnv)}{\gnv \in \Lev_{\max}S} \in \vE(d)$.
\end{corollary}
\section{The Forcing Notion} \label{sec:Forcing}
A finite sequence $\ordered{\gns_0, \dotsc, \gns_k} \in \leftexp{<\gw}{\ES}$ is said to be $\mo$-decreasing if it is increasing and $\ordered{\mo(\gns_0), \dotsc, \mo(\gns_k})$
	is non-increasing.
\begin{definition}
	A condition $f$ is in the forcing notion $\PP^*_f(\vE, \gee)$ if 
	$f$ is a function  $f \func d \to \leftexp{<\gw}{\ES}$ 
	such that:
	\begin{enumerate}
	\item
		$d \in [\gee]^{<\gl}$.
	\item
		$d \supseteq (\power{d}+1)$.
	\item
		For each $\ga \in d$,
			$f(\ga)$ is $\mo$-decreasing.
	\end{enumerate}
	Assume $f, g \in \PP_f^*(\vE, \gee)$ are conditions.
	We say $f$ is an  extension of $g$ ($f \leq^*_{\PP_f^*(\vE,\gee)} g$) if $f \supseteq g$.
\end{definition}
For a condition  $f \in \PP_f^*(\Vec{E})$ we will write $E_\gx(f)$ and $\vE(f)$ instead of 
					$E_\gx(\dom f)$ and $\Vec{E}(\dom f)$, respectively.
If $T \subseteq \OB(e)$ and $d \subseteq e$  then 
						$T \restricted d = \setof {\gn\restricted d}{\gn \in T}$.
\begin{definition}
	A condition $p$ is in the forcing notion $\PP^*(\Vec{E}, \gee)$ if 
		$p$ is of the form $\ordered{f^p, T^p}$, where	
						$f^p \in \PP^*_f(\Vec{E}, \gee)$,
					$T^p  \in \vE(f^p)$, and
		for each $\gn \in T^p$ 	and each $\ga \in \dom \gn$,
					$\max \mathring{f}^p(\ga) < \ogn(\gk)$.

Assume $p, q \in \PP^*(\Vec{E}, \gee)$ are conditions.
We say $p$ is a direct extension of $q$ ($p \leq^*_{\PP^*(\Vec{E}, \gee)} q$) if 
	$f^p \supseteq f^q$ and
	$T^p \restricted \dom f^q \subseteq T^q$.
We say $p$ is a strong direct extension of $q$ ($p \leq^{**}_{\PP^*(\Vec{E}, \gee)} q$) if
		$p$ is a direct extension of $q$ and $f^p =f^q$.
\end{definition}
Since $\gee$ and the sequence $\vE$ are fixed througout this work we designate
	$\PP^*(\Vec{E}, \gee)$ by  $\PP^*$.
\begin{definition}
A condition $p$ is in the forcing $\Bar{\PP}$ if 
	$p = \ordered{p_0, \dotsc, p_{n^p-1}}$, where $n^p < \gw$,
	 there is a  sequence $\ordof {\Vec{E}^p_i}{i < n^p}$
		such that each $\Vec{E}^p_i$ is a Mitchell increasing sequence of extenders,
			$\ordered{\crit(\Vec{E}^p_0), \dotsc, \crit(\Vec{E}^p_{n^p-1})}$ is strictly increasing,
				$\sup \setof {j_{E^p_{i,\gx}}(\crit \Vec{E}^p_i)} {\gx < \mo(\Vec{E}^p_i)} < \crit \Vec{E}^p_{i+1}$,
			$\gl(\Vec{E}^p_i) < \crit(\Vec{E}^p_{i+1})$,
			and for each $i < n^p$,	
						$p_i \in \PP^*(\Vec{E}^p_i, \gee^p_i)$.

Assume $p, q \in \Bar{\PP}$ are conditions.
We say $p$ is a direct extension of $q$ ($p \leq^*_{\Bar{\PP}} q$) if 
	$n^p = n^q$ and for each $i < n^p$, $p^i \leq^* q^i$.
We say $p$ is a strong direct extension of $q$ ($p \leq^{**}_{\Bar{\PP}} q$) if 
	$n^p = n^q$ and for each $i < n^p$, $p^i \leq^{**} q^i$.
\end{definition}
The following sequence of definitions leads to the definition of the order $\leq_{\Bar{\PP}}$
	(which is somewhat involved, hence the breakup to several steps).
If $\gn \in \OB(d)$ we let $\mo(\gn) = \mo(\gn(\gk))$.
\begin{definition}
Assume $f \func d \to \leftexp{<\gw}{\ES}$ is a function, $\gn \in \OB(d)$, and for each
		$\ga \in \dom \gn$, $\max \mathring{f}(\ga) < \ogn(\gk)$.
Define $f_{\ordered{\gn}\downarrow}$  and $f_{\ordered{\gn}\uparrow}$ as follows.
\begin{enumerate}
\item
		If $\mo(\gn)=0$ then $f_{\ordered{\gn}\downarrow} = \emptyset$.
		If $\mo(\gn) > 0$ then  $f_{\ordered{\gn}\downarrow}$ is the function $g$, where:
\begin{enumerate}
\item
	$\dom g = \ran \ogn$.
\item
	For each $\ga \in \dom \gn$,
		 $g(\ogn(\ga)) = \ordered{\gts_{k+1}, \dotsc, \gts_{n-1}}$,
		where $f(\ga) = \ordered{\gts_0, \dotsc, \gts_{n-1}}$ and
				$k<n$ is maximal such that
						$\mo(\gts_k) \geq \mo(\gn(\ga))$.
	Set $k = -1$ if there is no $k < n$ such that 			$\mo(\gts_k) \geq \mo(\gn(\ga))$.
\end{enumerate}
\item
	Define $f_{\ordered{\gn}\uparrow}$ to be the function $g$ where:
\begin{enumerate}
\item
	$\dom g = \dom f$.
\item
	For each $\ga \in \dom \gn$,
		 $g(\ga)= \ordered{\gts_0, \dotsc, \gts_{k}}$,
		where $f(\ga) = \ordered{\gts_0, \dotsc, \gts_{n-1}}$ and
				$k<n$ is maximal such that
						$\mo(\gts_k) \geq \mo(\gn(\ga))$.
	Set $k = -1$ if there is no $k < n$ such that 			$\mo(\gts_k) \geq \mo(\gn(\ga))$.
\end{enumerate}
\end{enumerate}
\end{definition}
\begin{definition}
The following definitions show how to reflect down a function $\gm  \in \OB(d)$
	using a larger function $\gn \in \OB(d)$.
\begin{enumerate}
\item
	Assume $\gm, \gn \in \OB(d)$, $\gm < \gn$, and $\mo(\gm) < \min(\mo(\gn), \ogn(\gk))$.
	Define the function $\gt = \gm \downarrow \gn \in \OB(\ran \ogn)$ by:
	\begin{enumerate}
	\item
		$\dom \gt = \setof{\ogn(\ga)}{\ga \in \dom \gm \intersect \dom \gn}$.
	\item
		For each $\gx \in \dom \gt$,
			$\gt(\gx) = \gm(\ga)$, were $\gx = \ogn(\ga)$.
	\end{enumerate}
\item
	Assume $T \subseteq \OB(d)$ and $\gn \in \OB(d)$.
	If $\mo(\gn) = 0$ then set $T_{\ordered{\gn}\downarrow} = \emptyset$.
	If $\mo(\gn) > 0$ then
			$T_{\ordered{\gn}\downarrow}= \setof {\gm \downarrow \gn} 
															{\gm \in T,\ \gm < \gn,\ \mo(\gm) < \min(\mo(\gn), \ogn(\gk))}$.
\end{enumerate}
\end{definition}
\begin{definition}
Assume $p \in \PP^*(\vE)$ and $\gn \in T^p$.
We define $p_{\ordered{\gn}\downarrow}$ as follows.
If $\mo(\gn) = 0$ then $p_{\ordered{\gn}\downarrow} = \emptyset$.
If $\mo(\gn) > 0$ then  
	 $p_{\ordered{\gn}\downarrow}$   is the condition $q \in \PP^*(\vbn)$ defined by setting
								 	$f^q = f^p_{\ordered{\gn}\downarrow}$  and
										$T^q = T^p_{\ordered{\gn}\downarrow}$.
Define $p_{\ordered{\gn}\uparrow}$ to be the condition $q \in \PP^*(\vE)$, where
					$f^q = f^p_{\ordered{\gn}\uparrow}$ and
						$T^q = T^p_{\ordered{\gn}}$.
Finally set 
		$p_{\ordered{\gn}} = \ordered{p_{\ordered{\gn}\downarrow}, p_{\ordered{\gn}\uparrow}}$.
\end{definition}
Of course for the above definition to make sense $T^p_{\ordered{\gn}\downarrow} \in \vbn(\ran \gn)$ should hold,
	which we prove in \cref{MakesSense}.

If $T \subseteq \OB(d)$ hen we let
	$\bT = \setof {\ordered{\gn_0, \dotsc, \gn_n}}
									{n<\gw,\ \gn_0, \dotsc, \gn_n \in T,\ \gn_0 < \dotsb < \gn_n}$.

\begin{definition} \label[definition]{RealOrder}
Assume $p, q \in \Bar{\PP}$.
We say $p$ is an extension of $q$ ($p \leq_{\Bar{\PP}} q$) if 
	 the following hold:
\begin{enumerate}
\item
	$n^p \geq n^q$.
\item
		$\setof{\Vec{E}^q_j}{j < n^q} \subseteq \setof {\Vec{E}^p_i}{i < n^p}$ and $\vE^q_{n^q-1} = \vE^p_{n^p-1}$.
\item
	For each $i < n^q$
		there is $\ordered{\gn_0, \dotsc, \gn_{k-1}} \in \bT^{q_i}$ such that
		$\ordered{p_{j_0+1}, \dotsc, p_{j_1}} \leq^* q_{i\ordered{\gn_0, \dotsc, \gn_{k-1}}}$,
			where $i$, $j_0$ and $j_1$, are being set as follows.
	Let $j_1 < n^p$ satisfy $\Vec{E}^p_{j_1} = \Vec{E}^q_{i}$.
	If $i = 0$ then set $j_0 = -1$.
	If $i > 0$ then let $j_0 < j_1$ satisfy  $\Vec{E}^p_{j_0}  = \Vec{E}^q_{i-1}$.
\end{enumerate}
\end{definition}
Finally we give the definition of the forcing notion we are going to work with:
\begin{definition}
$\PP(\Vec{E}, \gee) = \setof {q \leq_{\Bar{\PP}} p}{p \in \PP^*(\Vec{E}, \gee)}$.
The partial orders $\leq_{\PP(\Vec{E}, \gee )}$ and $\leq^*_{\PP(\Vec{E}, \gee)} $ are inherited from
	$\leq_{\Bar{\PP}}$  and $\leq^*_{\Bar{\PP}}$.
\end{definition}
Since $\gee$ and the sequence $\vE$ are fixed throughout this work we
		 will write $\PP$
	instead of $\PP(\Vec{E}, \gee)$ throughout this paper.

\Cref{MakesSense} is needed in order to show the forcing notion defined above makes sense.
\begin{claim} \label[claim]{MakesSense}
If $T \in \vE(d)$ then $X = \setof{\gn \in T}{T_{\ordered{\gn}\downarrow} \in \vbn(\ran\ogn)} \in \vE(d)$.
\end{claim}
\begin{proof}
We need to show 	$X \in \vE(d)$.
I.e., we need to show for each $\gx < \mo(\Vec{E})$,
							$X \in E_\gx(d)$.
Fix $\gx < \mo(\Vec{E})$.
We need to show
								$\mc_\gx(d_\gx) \in
											 j_{E_\gx}(X)$.
Hence it is enough showing
			$\mc_\gx(d) \in j_{E_\gx}(T)$ and
 			$ j_{E_\gx}(T)_{\ordered{\mc_\gx(d)}\downarrow}  \in \vE\restricted\gx(d)$.
Since $T \in \Vec{E}(d)$ we have $\mc_\gx(d) \in j_{E_\gx}(T)$.
So we are left with showing
					$j_{E_\gx}(T)_{\ordered{\mc_\gx(d)}\downarrow} \in \vE\restricted\gx(d)$.
From the definition of the operation $\downarrow$ we get
\begin{align*}
 j_{E_\gx}(T)_{\ordered{\mc_\gx(d)}\downarrow}= \begin{aligned}[t]
				\setof{\gm \downarrow \mc_\gx(d)}
									{\gm \in j_{E_\gx}(T), \ \gm < \mc_\gx(d),\ 
															\mo(\gm) < \min(\gk,\gx)}.
				\end{aligned}
\end{align*}
Consider $\gm \in j_{E_\gx}(T)$ such that $\gm < \mc_\gx(d)$.
There is $\gm^* \in T$ such that $\gm = j_{E_\gx}(\gm^*)$.
	Since for each $\gm^* \in T$  such that $\mo(\gm^*) < \gx$ we have
	$j_{E_\gx}(\gm^*) \downarrow \mc_\gx(d) = \gm^*$,
	 we get
	$j_{E_\gx}(T)_{\ordered{\mc_\gx(d)}\downarrow} =
				\setof{\gm \in T}{\mo(\gm) < \gx} \in \vE\restricted \gx(d)$.
\end{proof}

For each condition $p \in \PP$ 
	let $\PP/p = \setof {q \in \PP}{q \leq p}$.
It is immediate from the definitions above that for each $0<i<n^p-1$ the forcing notion $\PP/p$
	factors to $P_0 \times P_1$, where
		 $P_0 = \setof {q^0 \leq p^0}{q^0 \append p^1 \in \PP}$,
			$P_1 = \setof {q^1 \leq p^1}{p^0 \append q^1 \in \PP}$,
			$p^0 = \ordered{p_0, \dotsc, p_{i-1}}$, and
					 $p^1 = \ordered{p_i, \dotsc, p_{n^p-1}}$.
Together with the Prikry property (\cref{PrikryProperty}) and the closure of the direct order, 
	one can analyze the cardinal structure in $V^{\PP}$ straightforwardly.

If $e\supseteq d$ we define  $\gp^{-1}_{e,d}$ to be the inverse of
	the operation $\restricted d$, i.e., for each  $X \subseteq \OB(d)$ we let
	$\gp_{e,d}^{-1}(X) = \setof {\gn \in \OB(e)}{\gn \restricted d \in X}$.
	If $f, g \in \PP^*_f$ are conditions then we write
		$\gp_{f,g}^{-1}$ for $\gp_{\dom f, \dom g}^{-1}$.

We end this section with the analysis of the cardinal structure above
	$\gk$ in the generic extension:
	The cardinals between $\gk$ and $\gl$ are collapsed,
		and $\gl$ and the cardinals above it are preserved.
The properties of cardinals up to $\gk$  will be dealt with in later sections.
\begin{claim}
$\PP$ satisfies the $\gl^{+}$-cc.
\end{claim}
\begin{proof}
Begin with a family of conditions $\ordof {p^\gx}{\gx < \gl^{+}}$.
Without loss of generality we can assume $n^{p^{\gx_0}} = n^{p^{\gx_1}}$ for each $\gx_0, \gx_1 < \gl^{+}$.
Without loss of generality we can assume $\ordered{p^{\gx_0}_0, \dotsc, p^{\gx_0}_{n^{p^{\gx_0}}-2}} = \ordered{p^{\gx_1}_0, \dotsc, p^{\gx_1}_{n^{p^{\gx_1}}-2}}$ for each $\gx_0, \gx_1 < \gl^{+}$.
Thus, without loss of generality, we can assume $n^{p^{\gx}} = 1$  for each $\gx < \gl^{+}$.
By the $\gD$-system lemma we can assume $\setof {\dom f^{p^{\gx}}}{\gx < \gl^+}$ is a $\gD$-system 
	with kernel $d$.
	Since $\power{d} < \gl$ we can assume that for each $\gx_0, \gx_1 < \gl^+$ and $\ga \in d$,
		$f^{p^{\gx_0}}(\ga) = f^{p^{\gx_1}}(\ga)$.
		Fix $\gx_0 < \gx_1 < \gl^+$.
		Set $f = f^{p^{\gx_0}}\union f^{p^{\gx_1}}$,
			$T = \gp^{-1}_{f,f^{p^{\gx_0}}} T^{p^{\gx_0}}\intersect \gp^{-1}_{f,f^{p^{\gx_1}}}T^{p^{\gx_1}}$,
				and let $p = \ordered{f, T}$.
Then $p \leq p^{\gx_0}, p^{\gx_1}$.
\end{proof}
\begin{claim}
$\forces \formula {\text{There are no cardinals between $\gk$ and $\gl$}}$.
\end{claim}
\begin{proof}
Fix  a $V$-regular cardinal $\gt \in (\gk, \gl)$.
Fix a condition $p \in \PP$ such that 
	$\dom f^{p_{n^p-1}} \supseteq \gt \setminus \gk$ will hold.
Let $G \subseteq \PP$ be generic such that $p \in G$.
Set $C = \setof {\gnv \in \leftexp{<\gw}{T^{p_{n^p-1}}}}
							{p_{\ordered{\gnv}} \in G}$.
Then $\sup \setof {\sup(\gt \intersect  \bigunion \dom \gnv) }{\gnv \in C} = \gt$.
Since $\forces \formula{\power{C} \leq \gk}$ we get $p \forces \formula{ \cf \gt \leq \gk}$.
\end{proof}
Preservation of $\gl$ will be proved by a properness type argument (\cref{WeAreProper})
	for which we need some preparation.

We say the elementary substructure $N \subelem H_\gc$,
		where $\gc$ is large enough,
		 is $\gk$-internally approachable
	if there is an increasing continuous sequence of elementary substructures
		$\ordof{N_\gx}{\gx<\gk}$ such that 
			$N = \bigunion \setof{N_\gx}{\gx<\gk}$,
	for each $\gx < \gk$,
			$N_\gx \subelem H_\gc$,
				$\power{N_\gx} < \gl$,
					$N_{\gx+1} \supseteq \PSet_{\gk}(\power{N_{\gx}})$,
					$N_\gx \intersect \gl \in \On$,
						$\PP^*_f \in N_\gx$,
				$N_{\gx+1} \supseteq \leftexp{<\gk}{N_{\gx+1}}$, and
					$\ordof{N_{\gx'}}{\gx'<\gx} \in N_{\gx+1}$.

We say the pair $\ordered{N, f}$ is a good pair
	if $N \subelem H_\gc$ is a $\gk$-internally approachable elementary substructure
		and there is a sequence
					 $\ordof{\ordered{N_\gx, f_\gx}}{\gx<\gk}$ such that
				$\ordof {N_\gx}{\gx < \gk}$ witnesses the $\gk$-internal approachablity
						of $N$,
							$f = \bigunion \setof {f_\gx}{\gx < \gk}$,
							 $\ordof{f_\gx}{\gx<\gk}$ is a $\leq^*$-decreasing
								continuous sequence in $\PP^*_f$,
									 and for each $\gx < \gk$,
							$f_{\gx} \in \bigintersect
									\setof {D \in N_\gx}{D \text{ is a dense open subset of }
																\PP^*_f}$,
												$f_\gx \subseteq N_{\gx+1}$,
											 and
												$f_\gx \in N_{\gx+1}$.

Note that if $N \subelem H_\gc$ is an elementary substructure such that
	$\power{N} < \gl$, 
		$N \supseteq \leftexp{<\gk}{N}$, 
			$\PP_f^* \in N$, 
		$f \in \bigintersect \setof {D \in N}
								{D \text{ is a dense open subset of }\PP_f^*}$,
			$f \subseteq N$, and 
				$\ordered{\gn_0, \dotsc, \gn_{k-1}} \in N\intersect \OB(\dom f)$, then
				$f_{\ordered{\gn_0, \dotsc, \gn_{k-1}} } \in 
								\bigintersect \setof {D \in N}
												{D \text{ is a dense open subset of }\PP_f^*}$.
			
Hence if $\ordered{N, f}$ is a good pair and
	$\ordered{\gn_0, \dotsc, \gn_{k-1}} \in N \intersect \OB(\dom f)$, then
			$\ordered{N, f_{\ordered{\gn_0, \dotsc, \gn_{k-1}}}}$
				is a good pair also.

The following is immediate.
\begin{claim} \label[lemma]{GetGoodPair}
For each set $X$ and $f \in \PP^*_f$ there is 
	 a good pair$\ordered{N, f^*}$ 
	such that $ f^* \leq^* f$ and 
				$X,f \in N$.
\end{claim}
Assume $\gc$ is large enough and $N \subelem H_\gc$ is an elementary substructure such that
		$\PP \in N$.
		We say the condition $p \in N$ is $N$-generic if for each
				dense open subset $D \in N$ of $\PP$ we have
						$p \forces \formula{\VN{\PP} \intersect \CN{G} \intersect \VN{N} \neq \emptyset}$.
						
We say the forcing notion $\PP$ is $\gl$-proper if for an unbounded set of structures
	$N \subelem H_\gc$ such that $\PP \in N$ and $\power{N} < \gl$,
			and for each condition $p \in \PP \intersect N$ there is
				a stronger $N$-generic condition.

The followig lemma shows a property stronger than properness.
\begin{lemma}
Let $N\subelem H_\gc$ be a $\gk$-internally approachable structure,
	$\PP \in N$, 
		and
		$p \in N \intersect \PP$ a condition.
Then there is a direct extension $p^* \leq^* p$ such that
	for each dense open subset $D \in N$ of $\PP$
		the set $\setof{s \append p^*_{\ordered{\gn_0, \dotsc, \gn_{n-1}}\uparrow} \in D}
						{\ordered{\gn_0, \dotsc, \gn_{n-1}}  \in \leftexp{<\gw}{T}^{p^*},\ 
								s \leq^* p^*_{\ordered{\gnv}\downarrow}}$
								is predense below $p^*$.
Moreover, if $s \leq^* p^*_{\ordered{\gnv}\downarrow}$
				and
					$s \append p^*_{\ordered{\gn_0, \dotsc, \gn_{n-1}}\uparrow}
					 \in D$ then
	there is a weaker condition 
			$q \geq^* p^*_{\ordered{\gn_0, \dotsc, \gn_{n-1}}\uparrow}$
		such that $s \append q \in D \intersect N$.
\end{lemma}
\begin{proof}
	Let $\ordered{N, f^*}$ be a good pair such that $f^* \leq^* f^{p_{n^p-1}}$.
	Choose a set $T \in \vE(f^*)$ such that 
							$\ordered{f^*, T} \leq^* p_{n^p-1}$.
	Let $\ordof {D_\ga}{\ga < \power{N}}$ be an enumeration 
	of the dense open subsets	of $\PP$ appearing in $N$.
	Let $\ordof{\ordered{N_\gi,f_\gi}}{\gi < \gk}$ be a sequence witnessing
			$\ordered{N, f^*}$ is a good pair.
	For each $\ordered{\gn_0, \dotsc, \gn_{k-1}} \in \leftexp{<\gw}{T}$ 
		 construct the set $T^{\ordered{\gn_0, \dotsc, \gn_{k-1}}}$ as follows.

	Fix $\gnv = \ordered{\gn_0, \dotsc, \gn_{k-1}} \in \leftexp{<\gw}{T}$.
\newcommand{\cD}{\mathcal{D}}
\newcommand{\cE}{\mathcal{E}}
Let $\cD = \setof {D_\ga}{\ga \in \dom \gn_{k-1}}$.
Note $\cD \in N$ since $\power{\gn_{k-1}} < \gk$ and
		 $N \supseteq \leftexp{<\gk}{N}$.
For each $s \in \PP(\vbn_{k-1})$ and $D \in \cD$ define	the sets
		$D^\in_{\gnv,s,D}$,
					 $D^\incompatible_{\gnv,s,D}$, and
					 $D^*_{\gnv,s,D}$,
					  as follows:
	Let $g \in D^\in_{\gnv,s,D}$
		if $g \leq f^{p_{n^p-1}}$, $\dom g \supseteq \dom \gn_{k-1}$,
			 and $s \append \ordered{g_{\ordered{\gnv}},T'} \in D$ for some
			 			 $T' \in \vE(g)$.
	Let $h \in D^\incompatible_{\gnv,s,D}$ if
		$h \incompatible g$
		for each $g \in D^\in_{\ordered{\gnv},s,D}$.
	Set $D^*_{\gnv,s,D} = 
					D^\in_{\gnv,s,D} \union
				D^\incompatible_{\gnv,s,D}$.
		It is immediate  $D^\in_{\gnv,s,D}$ and
				$D^\incompatible_{\gnv,s,D}$ are open subsets of
						$\PP^*_f$ below$f^{p_{n^p-1}}$.
		Thus  $D^*_{\gnv,s,D}$ is a dense open subset of
			$\PP^*_f$ below $f^{p_{n^p-1}}$.
			Set $D^*_{\gnv} =
						 \bigintersect \setof {D^*_{\gnv,s,D}}
						 		{s \in \PP(\vbn_{k-1}),\  D \in \cD}$.
				Note $D^*_{\gnv} \in N$ is a dense open subset of $\PP^*_f$
						below $f^{p_{n^p-1}}$.
	Let $\gi < \gk$ be minimal 
			  such that 
						$\gnv,\cD, D^*_{\gnv} \in N_{\gi}$.
	Then $f_{\gi} \in D^*_{\gnv}\intersect N_{\gi+1}$.
	Thus for each $s \in \PP(\vbn_{k-1})$ and $D \in \cD$
	either there is a set $T^{\gnv,s,D} \in \vE(f_{\gi}) \intersect
							 N_{\gi+1}$
				 such that
	$s \append 
			\ordered{f_{\gi\ordered{\gnv}}, 
																		T^{\gnv,s,D}} \in D$
																					or
			  			  	$s \append \ordered{h,T''}\notin D$
			  			  		for each $h \leq^* f_{\gi\ordered{\gnv}}$ and $T''\in \vE(h)$.
			Set $T^\gn = \bigintersect \setof {T^{\gnv,s,D}}
								{s \in \PP(\vbn_{k-1}),\ D \in \cD,\ 
										s \append \ordered{f_{\gi\ordered{\gnv}},T^{\gn,s,D}}\in D }$.
										
	Set $T^* = \dintersect \setof {\gp^{-1}_{f^*,f_{\gi(\gnv)}}T^\gnv}
						{\gnv \in \leftexp{<\gw}{T}}$.
		Set $p^* = p\restricted n^p-1 \append \ordered{f^*, T^*}$.
		We claim $p^*$ satisfies the lemma.
		To show this fix a dense open subset $D \in N$ and a condition
				$q \leq p^*$.
				
		Let $\ga < \power{N}$ be such that $D = D_\ga$.
		Without loss of generality assume
				$q \in D$,
			$q_{n^q-1} \leq p^*_{\ordered{\gn_0, \dotsc, \gn_{k-1}}\uparrow}$,
						$\ordered{\gn_0, \dotsc, \gn_{k-1}} \in \leftexp{<\gw}{T}^*$,
					 and
					$\ga \in \dom \gn_{k-1}$.
				Set $s = q \restricted n^q - 1$.
				Thus  $q = s \append \ordered{f^{q_{n^q-1}}, T^{q_{n^q-1}}} \in D$.
				Let $\gi<\gk$ be minimal such that
						$\ordered{\gn_0, \dotsc, \gn_{k-1}}, \setof{D_\ga}{\ga \in \dom \gn_{k-1}}\in N_\gi$.
				Since $f^{q_{n^q-1}} \leq^* f_{\gi\ordered{\gn_0, \dotsc, \gn_{k-1}}}$
				 we must have
						$s \append \ordered{f_{\gi\ordered{\gnv}}, T^{\gnv,s,D}} \in D$,
						hence $s \append p^*_{\ordered{\gnv}\uparrow} \in D$.
		It is clear $q$ and $s \append p^*_{\ordered{\gnv}\uparrow}$
				 are compatible.
		In addition $s \append \ordered{f_{\gi\ordered{\gnv}}, T^{\gnv,s,D}} \in N$,
				thus we are done.
 \end{proof}
 \begin{corollary} \label[claim]{WeAreProper}
$\PP$ is $\gl$-proper.
\end{corollary}
\begin{corollary}
	$\forces \formula{\gl \text{ is a cardinal}}$.
\end{corollary}
%
%
%
%
%
%
%
%
%
%
%
%
%
%
%
%
%
%
%
%
%
%
%
%
%
%
%
\section{Dense open sets and measure one sets} \label[section]{sec:Dense}
In order to reduce clutter later on, 
	given a condition $p \in \PP^*$,
		we will say 	a tree is a $p$-tree instead of saying it is an $\vE(f^p)$-tree.
  If $S$ is a $p$-tree and $r$ is a function with domain $S$ then we define the function $\Vec{r}$
 	by setting for each
	$\gnv = \ordered{\gn_0, \dotsc, \gn_n} \in S$,
			$\Vec{r}(\gnv) = r(\gn_0) \append \dotsb \append r(\gn_0, \dotsc, \gn_{n})$.
A function $r$ is said to be a $\ordered{p, S}$-function if
		$S$ is a $p$-tree,
			for each  $\gnv \in \Lev_{<\max}S$, 
			$\Vec{r}(\gnv) \leq^{**} p_{\ordered{\gnv}\downarrow}$,
			and
			for each $\gnv \in \Lev_{\max}S$, $\Vec{r}(\gnv) \leq^{**} p_{\ordered{\gnv}}$.
\subsection{One of the measures suffices.}
The aim of this subsection is to prove \cref{GetPreDense}, which together with
	\cref{DenseHomogen} will allow the investigation  of the 
			cardinal structure  below $\gk$.
Note the proof	of \cref{DenseHomogen} depends on \cref{GetPreDense}.
The following lemma, which is quite technical,
		takes its core argument from the proof
		of the Prikry property for Radin forcing.
\begin{lemma} 
Assume $p \in \PP^*$ is a condition, $S$ is a $p$-tree of height one,
	and $r$ is a $\ordered{p,S}$-function.
Then there is a strong direct extension $p^* \leq^{**} p$
	such that 
			$\setof{ r(\gn)}{\ordered{\gn} \in S}$
				is predense below $p^*$.
\end{lemma}
\begin{proof}
	Define the functions $r_0$ and $r_1$, both with domain $S$, so that
		$r(\gn) = r_0(\gn) \append r_1(\gn)$ will hold for each $\ordered{\gn} \in S$.
	Fix $\gx < \mo(\Vec{E})$ so that $S \in E_{\gx}(f^p)$ will hold.
	We need to collect the information from the sets $T^{r_0(\gn)}$
			 and $T^{r_1(\gn)}$ into one set $T^*$.
	The information from the sets $T^{r_1(\gn)}$'s is collected by 
	setting
		$R = \dintersect_{\ordered{\gn} \in S} T^{r_1(\gn)}$.
By \cref{DiagonalIsBig} $R \in \vE(f^p)$.

	The information from the sets $T^{(r_0(\gn))}$'s is collected into the 
				set $T^*$ as follows.
		The set $T^*$ will be the union of the three sets $T^0, T^1$, and $T^2$,
			 which we construct now.
	The construction of $T^0$ is easy.	Set
		$T^{0} = T^{j_{E_\gx}(r_0)(\mc_\gx(f^p))}$.
	It is obvious $T^0 \in \vE\restricted \gx(f^p)$.
	
The constructin of $T^1$ is slightly more involved than the construction of $T^0$.	
	Set $T^{1\prime} = \setof {\ordered{\gn} \in S}
							{T^0_{\ordered{\gn} \downarrow}= T^{r_0(\gn)}}$.
 From the construction of $T^0$ it is clear $T^{1\prime} \in E_{\gx}(f^p)$.
For each $\gm \in T^0$ set $X(\gm) = \setof {\ordered{\gn} \in S}
			{\gm < \gn,\ \gm\downarrow \gn \in T^{r_0(\gn)}}$.
From the construction of $T^0$ 
	we get $X(\gm) \in E_\gx(f^p)$.
	Set $T^1 = \setof {\gn \in T^{1\prime}}
					{\forall \gm \in T^0\ (\gm < \gn \implies \gn \in X(\gm))}$.
	We show $T^1 \in E_\gx(f^p)$.
	Thus  we need to show $\mc_\gx(f^p) \in j_{E_\gx}(T^1)$.
	Since $\mc_\gx(f^p) \in j_{E_\gx}(T^{1\prime})$ it is enough to show that if
			$\gm \in j_{E_\gx}(T^0)$ and $\gm < \mc_\gx(f^p)$ then
											 $\mc_\gx(f^p) \in j_{E_\gx}(X)(\gm)$.
	So fix $\gm \in j_{E_\gx}(T^0)$ such that $\gm < \mc_\gx(f^p)$.
	Then $\power{\gm} < \gk$, $\dom \gm \subseteq j''_{E_\gx}(\dom f^p)$, and
				$\sup \ran \ogm < \gk$.
	Necessarily there is $\gm^* \in T^0$ such that $\gm = j_{E_\gx}(\gm^*)$.
	Hence $j_{E_\gx}(X)(\gm) = j_{E_\gx}(X(\gm^*)) \ni \mc_\gx(f^p)$,
										 by which we are done.
	
	We construct now the set $T^2$.
	For each $\gm \in R$ set
	$Y(\gm) = \setof{\gn\downarrow\gm  \in R_{\ordered{\gm}\downarrow}}
									{\gn \in T^1,\ 
									R _{
											\ordered{\gm}\downarrow
											 \ordered{\gn\downarrow\gm}\downarrow}\in \vbn(\dom \gn)
											}$.
Now let 
				$T^2 = \setof {\gm \in R}
				{\exists \gt < \mo(\gm)\ Y(\gm)		 \in \vbm_\gt(\dom \gm)}$.
		We show $T^2 \in E_\gz(f^p)$ for each $\gz > \gx$.
		We need  to show for each $\gz > \gx$, $\mc_\gz(f^p) \in j_{E_\gz}(T^2)$.
		Fix $\gz > \gx$.
		We show $\mc_\gz(f^p) \in j_{E_\gz}(T^2)$.
		It is enough to show there is $\gt < \gz$ such that 
					$j_{E_\gz}(Y)(\mc_\gz(f^p)) \in E_\gt(f^p)$.
We claim $\gx$ can serve as the needed $\gt < \gz$.
Thus it is enough to show
			$j_{E_\gz}(Y)(\mc_\gz(f^p)) \in E_\gx(f^p)$.
Hence we need to show
		\begin{multline*}
				\setof {
														\gn\downarrow \mc_\gz(f^p) \in j_{E_\gz}(R)_
																						{\ordered{\mc_\gz(f^p)}\downarrow}
					}
				{
								\gn \in j_{E_\gz}(T^1),\\
											\ j_{E_\gz}(R)_{
																\ordered{	\mc_\gz(f^p)}\downarrow
						 									\ordered{\gn\downarrow\mc_\gz(f^p)}\downarrow
												}
																							 \in \vbn(\dom \gn)
				} \in E_\gx(f^p).
\end{multline*}
Note $R^* = j_{E_\gz}(R)_{\ordered{\mc_\gz(f^p)}\downarrow} 
						\in \vE\restricted \gz(f^p)$, and 
		if $\gn \in j_{E_\gz}(T^1)$ and $\gn < \mc_\gz(f^p)$, then there is
		$\gn^* \in T^1$ such that $\gn = j_{E_\gz}(\gn^*)$.
Moreover, $\gn^* = \gn \downarrow \mc_\gz(f^p)$.
Hence it is enough to show
	\begin{align*}
				\setof {
														\gn^* \in R^*
					}
				{
								\gn^* \in T^1,
											\ R^*_{
								 									\ordered{\gn^*}\downarrow
												}
																							 \in \vbn^*(\dom \gn^*)
				} \in E_\gx(f^p).
\end{align*}
We are done since the last formula holds.

	Having constructed $T^0$, $T^1$, and $T^2$ we set
			$p^* = \ordered{f^p, T^* \intersect R}$.
	We will be done by showing $\setof {r(\gn)}{\gn \in S}$
								 is predense below $p^*$.
	Assume $q \leq p^*$. 
		We need to exhibit $\gn\in S$ so that $q \compatible r(\gn)$.
	We work as follows.
	Fix $\ordered{\gm_0, \dotsc, \gm_{n-1}}  \in \bT^{p^*}$ such that
		$q \leq^* p^*_{\ordered{\gm_0, \dotsc, \gm_{n-1}}}$.
	There are three cases to handle:
	\begin{enumerate}
	\item
	Assume there is $i<n$  such that
				 $\ordered{\gm_0, \dotsc, \gm_{i-1}} \in \bT^0$ and
											$\gm_i \in T^1$.
	The construction of $T^1$ yields 
				$\ordered{\gm_0, \dotsc, \gm_{i-1}} \in T^{r_0(\gm_i)}$ and
		the construction of $R$	 yields 
					$\ordered{\gm_{i+1}, \dotsc, \gm_{n-1}} \in \bT^{r_1(\gm_i)}$.
	Hence $r_{0}(\gm_i)_{\ordered{\gm_0, \dotsc, \gm_{i-1}}} \append
							 r_1(\gm_i)_{\ordered{\gm_{i+1}, \dotsc, \gm_{n-1}}}$
		and $q$ are $\leq^*$-compatible, by which this case is done.
		\item
	Assume  $\ordered{\gm_0, \dotsc, \gm_{n-1}} \in \bT^0$.
	By the construction of $T^1$ the set
			 $X = \setof{\gn \in T^1}
			 			{\ordered{\gm_0, \dotsc, \gm_{n-1}}\downarrow \gn
			 						 \in \bT^0_{\ordered{\gn}\downarrow}} \in E_\gx(f^p)$.
	Choose $\gn^* \in T^{q_{n^{n^q}-1}}$ such that 
		$\gn = \gn^* \restricted f^p \in X$.
	Then $q_{\ordered{\gn^*}} \leq^* 
					p_{\ordered{\gm_0, \dotsc, \gm_{n-1}, \gn}}$.
	Now we can procced as in the first case above.
\item	
	The last case is when there is $i<n$  such that 
					$\ordered{\gm_0, \dotsc, \gm_{i-1}} \in \bT^0$ and
											$\gm_i \notin T^0 \union T^1$.
	By the construction of $T^2$ there is $\gt < \mo(\gm_i)$ such that
			 $Y = Y(\gm_i)  \in \vbm_{i\gt}(\dom \gm_i)$.
Hence there are $\gm_{i\gt}(\dom \gm_i)$-many  $\gn \downarrow \gm_i$
				 such that
				 $\gn \in T^1$, 
					$\gn\downarrow\gm_i \in T^*_{\ordered{\gm_i}\downarrow}$  and
					$T^*_{\ordered{\gm_i}\downarrow \ordered{\gn}\downarrow \gm_i} \in
								 \gn(\dom \gn)$.

	Thus there is  $\gs^* \in T^{q_i}$ such that 
						$\gs = \gs ^*\restricted \dom f^{p_i} \in Y$,
				where $\gs = \gn \downarrow \gm_i$ and $\gn \in T^1$.
	Thus $q_{\ordered{\gs^*}} \leq^*
					 p_{\ordered{\gm_0, \dotsc, \gm_{i-1}, \gn, \gm_i, \dotsc, \gm_{n-1}}}$
			and we can proceed as in the first case above.
	\end{enumerate}
\end{proof}%
\begin{corollary}  \label[corollary]{GetPreDense1}
Assume $p \in \PP$ is a condition, $S$ is a $p_{n^p-1}$-tree of height one,
	and $r$ is a $\ordered{p_{n^p-1},S}$-function.
Then there is a strong direct extension $p^* \leq^{**} p$
	such that $p^* \restricted n^p-1 = p \restricted n^p-1$ and
			$\setof{p \restricted n^p-1 \append  r(\gn)}{\ordered{\gn} \in S}$
				is predense open below $p^*$.
\end{corollary}
Generalize the notions of $p$-tree and $\ordered{p, S}$-function
		to arbitrary condition $p \in \PP$
		as follows.
		By recursion we say the tree $S$ is a $p$-tree if there is $n<\gw$ for which following hold:
\begin{enumerate}
\item
	$\Lev_{< n}(S)$ is a $p\restricted n^p-1$-tree.
\item
		For each $\gnv \in \Lev_{n-1}(S)$, $S_{\ordered{\gnv}}$ is a $p_{n^p-1}$-tree.
\end{enumerate}
Let $p \in \PP$ be an arbitrary condition.
By recursion we say the function $r$ is a $\ordered{p,S}$-function if there is $n < \gw$ such that:
\begin{enumerate}
	\item
		$S$ is a $p$-tree.
	\item
		$\Lev_{< n}(S)$ is a $p \restricted n^p-1$-tree.
\item
		$ r\restricted \Lev_{< n}S$ is a $\ordered{\Lev_{< n}S, p\restricted n^p-1}$-function.
	\item
		For each $\gnv \in \Lev_{n-1}(S)$
			 the function $s$ with domain $S_{\ordered{\gnv}}$, 
			 	define by setting $s(\gmv) = r(\gnv\append\gmv)$, is a
			 		$\ordered{p_{n^p-1}, S_{\ordered{\gnv}}}$-function.
\end{enumerate}
\begin{claim} \label[claim]{GetPreDense}
Assume $p \in \PP$ is a condition, $S$ is a $p$-tree,
	and $r$ is a $\ordered{p,S}$-function.
Then there is a strong direct extension $p^* \leq^{**} p$
	such that 
			$\setof{ \Vec{r}(\gnv)}{\gnv \in S}$
				is predense below $p^*$.
\end{claim}
\begin{proof}
	If $S$ is a $p_{n^p-1}$-tree  then we are done by \cref{GetPreDense1}.
	Thus assume there is $n<\gw$ such that
		$\Lev_{<n}S$ is a $p\restricted n^p-1$-tree.
	Construct the strong direct extension
			 $q(\gnv) \leq^{**} p_{\ordered{\gnv}\uparrow}$
			and the 
						$\ordered{p_{\ordered{\gnv}\uparrow}, S_{\ordered{\gnv}}}$-function
									$s_{\gnv}$
									 for each  $\gnv \in \Lev_{n-1}S$ as follows.
	For each $\gnv \in \Lev_{n-1}S$ let $s_\gnv$ be the function with domain
				 $S_{\ordered{\gnv}}$ defined by setting
		$s_\gnv(\gmv) = r(\gnv\append\gmv)$ for each $\gmv \in S_{\ordered{\gnv}}$.
	By \cref{GetPreDense1} there is a strong direct extension $q(\gnv) \leq^{**} p_{\ordered{\gnv}\uparrow}$
		such that $\setof {s_\gnv(\gmv)}{\gmv \in \Lev_{\max}(S(\gnv))}$ is
								predense below $q(\gnv)$.
Let $q \leq^{**} p_{n^p-1}$ be a strong direct extension satisfying
		$q \leq^{**} q(\gnv)$ for each $\gnv \in \Lev_{n-1}(S)$.
		Hence $\setof {s_\gnv(\gmv)}{\gmv \in \Lev_{\max}(S(\gnv))}$ is
								predense below $q$ for each $\gnv \in \Lev_{n-1}S$.
	Hence $\setof {\Vec{r}(\gnv) \append s_\gnv(\gmv)}
									{\gmv \in \Lev_{\max}S(\gnv)}$ is 
														predense below $\Vec{r}(\gnv) \append q$.
	Let $p^* \leq^{**} p$ be a strong direct extension such that
			$p^*_{n^p-1} = q$ and $p^* \restricted n^p-1 \leq^{**} p\restricted n^p-1$
					is a strong direct extension constructed by recursion so as to satisfy
							$\setof {\Vec{r}(\gnv)}{\gnv \in \Lev_{n-1}S}$ is predense
									below $p^*\restricted n^p-1$.
		Necessarily $\setof {\Vec{r}(\gnv)}{\gnv \in \Lev_{\max}S}$ is
				predense below $p^*$
\end{proof}
\subsection{Dense open sets and direct extensions}
In this subsection we prove \cref{DenseHomogen}, which is the basic tool to be used in the next section
	to analyse the properties of the cardinal $\gk$ and the 
		cardinal structure below it.

An essential obstacle in the extender based Radin forcing in comparison
	to the plain extender forcing is that while in the
		later forcig notion if we have two direct extensions $q,r\leq^* p$ then
				$q$ and $r$ are compatible,
	in the former forcing notion this does not hold.
This usually entails some inductions, taking place inside elementary
	substructures, which  construct long
	increasing seqeunce of conditions from $\PP^*_f$, which
		at the end will be combined into one conditions.
This method breaks if the elementary substructures in question are not
	closed enough (which is our case if we want to handle $\gl$
		successor of singular).
The point of \cref{DenseHomogenOneBlock}	is to show
	how we can construction a condition $p$ such that
		if a direct extension $q \leq^* p$ has some favorable circumstances
			then the condition  $p$ will suffice for this circumstances.
This will enable us to work more like in a plain Radin forcing.

So as we just pointed out,
	we aim  to prove \cref{DenseHomogenOneBlock}.
This lemma is proved by recursion with the non-recursive case being 
	\cref{DenseHomogenOneBlockCase0}.
	Since the notation in \cref{DenseHomogenOneBlock} is kind of hairy
			we present the cases $k=1$ and $k = 2$ in \cref{DenseHomogenOneBlockCase1}
				and \cref{DenseHomogenOneBlockCase2}, respectively.
	\begin{lemma} \label[lemma]{DenseHomogenOneBlockCase0}
Assume $\ordered{N, f^*}$ is a good pair and
	$D \in N$ is a dense open set.
Let $p \in \PP$ be a condition such that $f^{p_{n^p-1}} = f^*$.
If there is an extension $s \leq p \restricted n^p-1$
	 and 
		a direct extension $q \leq^*  p_{n^p-1}$ such that 
					$s \append q \in D$
	 then  there is a set $T^* \in \vE(f^*)$ such that
	 			$\ordered{f^*, T^*} \leq^{**} p_{n^p-1}$ and
			$s  \append  \ordered{f^*, T^*} \in D$.
\end{lemma}
\begin{proof}
Assume $s \leq p \restricted n^p-1$, 
						 $q \leq^* p_{n^p-1}$, and  $s \append q \in D$.
Set $D^\in = \setof {g}
			{\exists T\in \vE(g)\  \ 
							s \append \ordered{g, T} \in D}$ and
	$D^\incompatible = \setof {g }{\forall h \in D^\in\ g \incompatible h}$.
Then 	 $D^\incompatible \in N$ is open by its definiton and
	$D^\in \in N$ is open since $D$ is open.
The set $D^* = D^\in \union D^\incompatible \in N$
			 is dense open, hence
			$f^* \in D^*$.
Since $f^* \geq f^{q} \in D^\in$ we get $f^* \notin D^\incompatible$,
	thus $f^* \in D^\in$.
\end{proof}
\begin{lemma} \label[lemma]{DenseHomogenOneBlockCase1}
Assume $\ordered{N, f^*}$ is a good pair,
	$D \in N$ is a dense open set,
					and  $p \in \PP$ is a condition such that $f^{p_{n^p-1}} = f^*$.
If there is an extension $s \leq p\restricted n^p-1$
	and $\gx < \mo(\vE)$ 
		such that
			$\setof {\gn \in T^{p_{n^p-1}}}
				{\exists q \leq^* p_{n^p-1\ordered{\gn}}\ s \append q \in D} \in E_\gx(f^*)$, then
		there is
			a $p_{n^p-1}$-tree $S$ of height one,
		 and 
				a $\ordered{p_{n^p-1},S}$-function r, such that
						for each $\ordered{\gn} \in S$,
	$s \append r(\gn) \in D$.
\end{lemma}
\begin{proof}
Assume $X =  \setof {\gn \in T^{p_{n^p-1}}}
						{\exists q \leq^* p_{n^p-1\ordered{\gn}}\ s \append q \in D}
	 \in E_\gx(f^*)$.
	 Set 
	 		\begin{multline*}
	 				D^\in = \setof {g}
					{\exists T\in \vE(g),\   
							\text{there is a $\ordered{g,T}$-tree $S$ of height one and} \\
										\text{a 	$\ordered{\ordered{g,T},S}$-function $r$ such that } 
										\forall \ordered{\gn} \in S\ 
												s \append r(\gn)\in D}
				\end{multline*}
				 and
	$D^\incompatible = \setof {g }{\forall h \in D^\in\ g \incompatible h}$.
Then 	 $D^\incompatible \in N$ is open by its definiton and
	$D^\in \in N$ is open since $D$ is open.
The set $D^* = D^\in \union D^\incompatible \in N$
			 is dense open, hence
			$f^* \in D^*$.
For each $\gn \in X$ fix 
		a direct extension $t(\gn) \leq^* p_{n^p-1\ordered{\gn}\downarrow}$,  and
			a direct extension $q(\gn) \leq^* p_{n^p-1\ordered{\gn}\uparrow}$ 
				such that $s \append t(\gn) \append q(\gn) \in D$.
Since $\ordered{N, f^*_{\ordered{\gn}\uparrow}}$ 
		is a good pair we get by the previous lemma
								 a set $T(\gn) \in \vE(f^*_{\ordered{\gn}\uparrow})$
			satisfying
				$\ordered{f^*_{\ordered{\gn}\uparrow}, T(\gn)} \leq^{**} p_{n^p-1\ordered{\gn}\uparrow}$
					and
						$s \append t(\gn) \append \ordered{f^*_{\ordered{\gn}\uparrow}, T(\gn)} \in D$.
										
Set $g = f^* \union f^{j_{E_\gx}(t)(\mc_\gx(f^*))}$.
Set $X^* = \gp^{-1}_{g,f^*}(X)$.
By removing a measure zero set from $X^*$ we can assume for each $\gn \in X^*$,
	$g_{\ordered{\gn}\downarrow} = f^{t(\gn\restricted \dom f^*)}$.
	Choose a set $T \in \vE(g)$ such that $\ordered{g,T} \leq^* p_{n^p-1}$.
	Define the funcrion $r$ with domain $X^*$ by setting for each $\gn \in X^*$,		
			$r(\gn) = \ordered{g_{\ordered{\gn}\downarrow}, T^{(t(\gn\restricted \dom f^*))} \intersect T_{\ordered{\gn}\downarrow}} \append 
			\ordered{g_{\ordered{\gn}\uparrow}, \gp^{-1}_{g,f^*}T(\gn) \intersect T}$.
			Note $r(\gn) \leq^{**} \ordered{g, T}_{\ordered{\gn}}$, thus
					$r$ is a $\ordered{g,X^*}$-function.
Since $D$	 is open we get for each $\gn \in X^*$,
		$s \append r(\gn) \in D$.
Thus $g \in D^\in$.
Since $g \leq f^* \in D^*$ we get $f^* \in D^\in$.
\end{proof}
\begin{lemma} \label[lemma]{DenseHomogenOneBlockCase2}
Assume $\ordered{N, f^*}$ is a good pair,
				$D \in N$ is a dense open set,
						and $p \in \PP$ is a condition such that $f^{p_{n^p-1}} = f^*$.
If there is $s \leq p\restricted n^p-1$ such that
	$\setof {\ordered{\gn_0, \gn_1} \in \leftexp{2}{T}^{p_{n^p-1}}}
					{\exists q \leq^* p_{n^p-1\ordered{\gn_0, \gn_1}}\ s \append q \in D}$
					is an $\vE(f^*)$-tree, then there is
			a $p_{n-1}$-tree $S$ of height two,
				and a $\ordered{p_{n^p-1},S}$-function r such that
			for each $\ordered{\gn_0, \gn_1} \in S$, 
	$s \append \Vec{r}(\gn_0, \gn_1) \in D$.
\end{lemma}
\begin{proof}
Assume $X =  \setof {\ordered{\gn_0, \gn_1} \in \leftexp{2}{T}^{p^*_{n^p-1}}}
										{\exists s \leq p\restricted n^p-1\ \exists q \leq^* p_{n^p-1\ordered{\gn_0, \gn_1}}\ q \in D}$ is
												an $\vE(f^*)$-tree.
												 Set 
	 		\begin{multline*}
	 				D^\in = \setof {g}
					{\exists T\in \vE(g),\  
							\text{there is a $\ordered{g,T}$-tree $S$ of height two and} \\
										\text{a 	$\ordered{\ordered{g,T},S}$-function $r$ such that }\\ 
										\forall \ordered{\gn_0, \gn_1} \in S\ 
												s \append \Vec{r}(\gn_0, \gn_1)\in D}
				\end{multline*}
				 and
	$D^\incompatible = \setof {g }{\forall h \in D^\in\ g \incompatible h}$.
Then 	 $D^\incompatible \in N$ is open by its definiton and
	$D^\in \in N$ is open since $D$ is open.
The set $D^* = D^\in \union D^\incompatible \in N$
			 is dense open, hence
			$f^* \in D^*$.
For each $\ordered{\gn_0, \gn_1} \in X$ fix 
			a direct extension  
					$t(\gn_0, \gn_1) \leq^* p_{n^p-1\ordered{\gn_0}\downarrow }$,
					a direct extension  
					$t_0(\gn_0, \gn_1) \leq^* p_{n^p-1\ordered{\gn_0}\uparrow\ordered{\gn_1}\downarrow}$,
							and a direct extension 
						$q(\gn_0, \gn_1) \leq^* p_{n^p-1\ordered{\gn_0, \gn_1}\uparrow}$ 
				such that $s \append t(\gn_0, \gn_1) \append 
										t_0(\gn_0, \gn_1) \append q(\gn_0, \gn_1) \in D$.
					For each $\gn_0 \in \Lev_0(X)$ we can remove a measure zero set from
								$\Suc_X(\gn_0)$ so that we can assume
											there is a direct extension 
															$t(\gn_0) \leq^* p^*_{n^p-1\ordered{\gn_0} \downarrow}$
											such that
															$t(\gn_0) = t(\gn_0, \gn_1)$ for each $\gn_1 \in \Suc_X(\gn_0)$.
				By the previous lemma there is 
					a $p_{n^p-1\ordered{\gn_0}\uparrow}$-tree $S(\gn_0)$ of height one,
						and a $\ordered{p_{n^p-1\ordered{\gn_0}\uparrow}, S(\gn_0)}$-function $r_{\gn_0}$
								satisfying for each $\gn_1 \in S(\gn_0)$,
						$s \append t(\gn_0) \append r_{\gn_0}(\gn_1) \in D$.
						
Set $g = f^* \union f^{j_{E_\gx}(t)(\mc_\gx(f^*))}$, where
		$\Lev_0(X) \in E_\gx(f^*)$.
Set $X^* = \setof {\ordered{\gn_0, \gn_1}}
						{\gn_0 \in \gp^{-1}_{g,f^*}\Lev_0(X), \ 
											\gn_1 \in \gp^{-1}_{g,f^*}S(\gn_0\restricted \dom f^*)}$.
By removing a measure zero set from $\Lev_0(X^*)$ we can assume for each $\gn_0 \in X^*$,
	$g_{\ordered{\gn_0}\downarrow} = f^{t(\gn_0\restricted \dom f^*)}$.
	Choose a set $T \in \vE(g)$ such that $\ordered{g,T} \leq^* p^*_{n^p-1}$.
	For each $\gn_0 \in \Lev_0(X^*)$ let $r'_{\gn_0}$ be
		the function with domain $\Suc_{X^*}(\gn_0)$ defined by shrinking the trees
				in $r_{\gn_0}$ so that both
					$r'_{\gn_0}(\gn_1) \leq^{**} 
									r_{\gn_0\restricted \dom f^*}(\gn_1 \restricted \dom f^*)$ and
							$r'_{\gn_0}(\gn_1) \leq^{**} \ordered{g,T}_{\ordered{\gn_0}\uparrow\ordered{\gn_1}}$
									will hold for each $\gn_1 \in \Suc_{X^*}(\gn_0)$.
	Define the function $r$ with domain $X^*$ by setting 
					for each $\ordered{\gn_0, \gn_1} \in X^*$,		
									$r(\gn_0) = \ordered{g_{\ordered{\gn_0}\downarrow},
												 T_{\ordered{\gn_0}\downarrow} \intersect
												 				\gp^{-1}_{g_{\ordered{\gn_0}\downarrow},f^{t(\gn_0\restricted \dom f^*)}}T^{(t(\gn_0  \restricted \dom f^*))}
												 											}$
							and
								$r(\gn_0, \gn_1) =
														 	r'_{\gn_0}
														 			(\gn_1)$.

			Note $\Vec{r}(\gn_0, \gn_1) \leq^{**} \ordered{g, T}_{\ordered{\gn_0, \gn_1}}$, thus
					$r$ is a $\ordered{g,X^*}$-function.
Since $D$	 is open we get
		$s \append \Vec{r}(\gn_0, \gn_1) \in D$ for each $\ordered{\gn_0, \gn_1} \in X^*$.
Thus $g \in D^\in$.
Since $g \leq f^* \in D^*$ we get $f^* \in D^\in$.
\end{proof}
As discussed earlier, the following lemma is the intended one, with the previous
	ones serving as an introduction to the technique used in the proof.
\begin{lemma} \label[lemma]{DenseHomogenOneBlock}
Assume $\ordered{N, f^*}$ is a good pair,
					 $k<\gw$,
							$D \in N$ is a dense open set,
								and $p \in \PP$ is a condition such that $f^{p_{n^p-1}} = f^*$.
If there is $s \leq p\restricted n^p-1$ such that
	$\setof {\gnv \in \leftexp{k}{T}^{p_{n^p-1}}}
					{	\exists q \leq^* p_{n^p-1\ordered{\gnv}}\ s \append q \in D}$
					is an $\vE(f^*)$-tree,
	then 		there is 
	a $p_{n^p-1}$-tree $S$ of height $k$,
			and a $\ordered{p_{n^p-1},S}$-function r such that
			for each $\gnv \in \Lev_{\max}S$, 
	$s \append \Vec{r}(\gnv) \in D$.
\end{lemma}\begin{proof}
Assume $X =  \setof {\ordered{\gm} \append \gnv \in \leftexp{k}{T}{p_{n^p-1}}}
										{\exists q \leq^* p_{n^p-1\ordered{\ordered{\gm}\append \gnv}}\ s \append q \in D}$ is
												an $\vE(f^*)$-tree.
For each $\ordered{\gm} \append \gnv \in X$ fix a direct extension 
						$t(\gm \append \ordered{\gnv}) \leq^*  p_{n^p-1\ordered{\gm}\downarrow}$
							and a direct extension
						$q(\gm \append \gnv) \leq^* p_{n^p-1\ordered{\gm}\uparrow\ordered{ \gnv}}$ 
				such that $s \append t(\ordered{\gm} \append \gnv)\append q(\ordered{\gm} \append \gnv) \in D$.
					For each $\gm \in \Lev_0(X)$ we can remove a measure zero set from
								$X_{\ordered{\gm}}$ so that we will have a direct extension
											$t(\gm) \leq^* p_{n^p-1\ordered{\gm}\downarrow}$ such that
															$t(\gm) = t(\ordered{\gm} \append \gnv)$
																		 for each $\gnv \in X_{\ordered{\gm}}$.
				By recursion there is 
					a $p_{n^p-1\ordered{\gm}\uparrow}$-tree $S(\gm)$ of height $k-1$,
						and a $\ordered{p_{n^p-1\ordered{\gm}\uparrow}, S(\gm)}$-function $r_{\gm}$
								satisfying for each $\gnv \in S(\gm)$,
						$s \append t(\gm) \append r_{\gm}(\gnv) \in D$.
										
Set $g = f^* \union f^{j_{E_\gx}(t)(\mc_\gx(f^*))}$, where
		$\Lev_0(X) \in E_\gx(f^*)$.
Set $X^* = \setof {\ordered{\gm} \append \gnv}
						{\gm \in \gp^{-1}_{g,f^*}\Lev_0(X), \ 
											\gnv \in \gp^{-1}_{g,f^*}(S(\gm\restricted \dom f^*))}$.
By removing a measure zero set from $\Lev_0(X^*)$ we can assume for each $\gm \in X^*$,
	$g_{\ordered{\gm}\downarrow} = f^{t(\gm \restricted \dom f^*)}$.
	Choose a set $T \in \vE(g)$ such that $\ordered{g,T} \leq^* p_{n^p-1}$.
	For each $\gm \in \Lev_0(X^*)$ let $r'_{\gn_0}$ be
		the function with domain $X^*_{\ordered{\gm}}$ defined by shrinking the trees
				in $r_{\gm}$ so that both
					$\Vec{r}'_{\gm}(\gnv) \leq^{**} 
									\Vec{r}_{\gm\restricted \dom f^*}(\gnv \restricted X_{\ordered{\gm} \restricted \dom f^*})$ and
							$r'_{\gm}(\gnv) \leq^{**} \ordered{g,T}_{\ordered{\gm}\uparrow\ordered{\gnv}}$
									will hold for each $\gnv \in X^*_{\ordered{\gm}}$.
	Define the function $r$ with domain $X^*$ by setting 
					for each $\ordered{\gm} \append \gnv \in X^*$,		
									$r(\gm) = \ordered{g_{\ordered{\gm}\downarrow},
												 T_{\ordered{\gm}\downarrow} \intersect
												 				\gp^{-1}_
												 					{g_{\ordered{\gm}\downarrow},f^{t(\gm\restricted \dom f^*)}
												 														}T^{(t(\gm  \restricted \dom f^*))}
												 											}$
							and
								$r(\ordered{\gm} \append \gnv) =
														 	r'_{\gm}
														 			(\gnv)$.

			Note $\Vec{r}(\ordered{\gm} \append  \gnv) \leq^{**} 
						\ordered{g, T}_{\ordered{\ordered{\gm} \append \gnv}}$, thus
					$r$ is a $\ordered{g, X^*}$-function.
Since $D$	 is open we get
		$s \append \Vec{r}(\ordered{\gm} \append \gnv) \in D$ for each 
							$\ordered{\ordered{\gm} \append \gnv} \in X^*$.
Thus $g \in D^\in$.
Since $g \leq f^* \in D^*$ we get $f^* \in D^\in$.
\end{proof}
	\begin{lemma}
		Assume $\ordered{f, T} \in \PP$ is a condition,
				$k < \gw$,
				 and $S \subseteq \leftexp{k}{T}$
			is not an $\vE(f)$-tree.
			Then there is a set $T^* \in \vE(f)$ such that $\ordered{f, T^*} \leq^* \ordered{f, T}$ and
						$\leftexp{k}{T}^{*} \intersect S = \emptyset$.
	\end{lemma}
	\begin{proof}
	 By removing measure zero sets from the levels of $S$ we can find $n < k$ so that the following will hold:
	 \begin{enumerate}
	 \item
		For each $l < n$ and $\ordered{\gn_0, \dotsc, \gn_{l-1}} \in S$,
				$\Suc_S(\gn_0, \dotsc, \gn_l) \in E_\gx(f)$
						for some $\gx < \mo(\vE)$.
	\item
		For each $\ordered{\gn_0, \dotsc, \gn_{n-1}} \in S$,
							$\Suc_S(\gn_0, \dotsc, \gn_{n-1}) \notin E_\gx(f)$
										for each $\gx  < \mo(\vE)$.
	\end{enumerate}
	Shrink $T$ so that 
			$\setof{\ordered{f,T}_{\ordered{\gn_0, \dotsc, \gn_{n-1}}}}
								{\ordered{\gn_0, \dotsc, \gn_{n-1}} \in \Lev_n(S)}$
								 is predense below   $\ordered{f,T}$.
		We are done by setting 
				$A = \dintersect \setof {T \setminus \Suc_{S_n}(\gn_0, \dotsc, \gn_{n-1})	}
						{\ordered{\gn_0, \dotsc, \gn_{n-1}} \in S_n}$
			and $T^* = T \intersect A$.
	\end{proof}
	\begin{corollary}
Assume $\ordered{N, f^*}$ is a good pair,
	$D \in N$ is a dense open set,
		and $p \in \PP$ is a condition such that $f^{p_{n^p-1}} = f^*$.
Assume $s \leq p \restricted n^p-1$.
Then one and only one of the following holds:
\begin{enumerate}
\item
	There is a $p_{n^p-1}$-tree S,
		and
			a $\ordered{p_{n^p-1}, S}$-function $r$, such that
					for each $\gnv \in \Lev_{\max} S$,
										$s  \append \Vec{r}(\gnv) \in D$.
\item
	There is a set $T^* \in \vE(f^*)$ such that
			$\ordered{f^*, T^*} \leq^{**} p_{n^p-1}$ and
		for each $\gnv \in \leftexp{<\gw}{T}^*$ and
				$q \leq^* \ordered{f^*, T^*}_{\ordered{\gnv}}$,
						$s \append q \notin D$.
	\end{enumerate}
\end{corollary}
	It is about time we get rid of the conditional appearing in
		the former statements  show we
	have densely many times $p_{n^p-1}$-trees and functions.
\begin{claim}
Assume 	$D \in N$ is a dense open set,
		and $p \in \PP$ is a condition.
Then there is an extension $s \leq p \restricted n^p-1$ and $k < \gw$ so that
				$\setof {\gnv \in \leftexp{k}{T}^{p_{n^p-1}}}
						{\exists q\leq^* p_{n^p-1\ordered{\gnv}}\ s \append q \in D}$
						is an $\vE(f^p)$-tree.
\end{claim}
\begin{proof}
Towards contradiction assume the claim fails.
Then for each $s \leq p\restricted n^p - 1$ and $k<\gw$ the set
	$S(s,k) = \setof {\gnv \in \leftexp{k}{T}^{p_{n^p-1}}}
						{\exists q\leq^* p_{n^p-1\ordered{\gnv}}\ s \append q \in D}$
						is not an $\vE(f^p)$-tree.
Thus there is a set $T(s, \gk) \in \vE(f^{p_{n^p-1}})$ satisfying
	$\leftexp{k}{T}(s, k) \intersect S(s,k) = \emptyset$.
	Set $T^* = \bigintersect \setof {T(s,k)}{k<\gw,\  s\leq p\restricted n^p-1}$.
	Consider the condition $p^* = p\restricted n^p-1 \append \ordered{f^{p_{n^p-1}}, T^*}$.
	By the density of the set $D$ there is an extension $s \leq p \restricted n^p-1$,
			$\gnv \in \leftexp{k}{T}^*$,  and
				$q \leq^* p^*_{n^p-1\ordered{\gnv}}$,
				such that $s \append q \in D$.
	Hence $\gnv \in S(s, k)$.
		However $\leftexp{k}{T^*} \intersect S(n,k) = \emptyset$, contradiction.
\end{proof}
\begin{corollary}
Assume $\ordered{N, f^*}$ is a good pair,
	$D \in N$ is a dense open set,
 and  $p \in \PP$ is a condition such that $f^{p_{n^p-1}} = f^*$.
Then there is a maximal antichain $A$ below $p \restricted n^p-1$ such that
	for each $s \in A$
		there is a $p_{n^p-1}$-tree S
		and
			a $\ordered{p_{n^p-1}, S}$-function $r$, such that
					for each $\gnv \in \Lev_{\max} S$,
										$s  \append \Vec{r}(\gnv) \in D$.
\end{corollary}
\begin{corollary} \label[corollary]{DenseHomogen}
Assume 
	$D \in N$ is a dense open set
			 and  $p \in \PP$ is a condition.
 Then there is a direct extension $p^* \leq^* p$,
 			a $p^*$-tree $S$, and
 			a $\ordered{p^*, S}$-function $r$ such that
 					for each $\gnv \in \Lev_{\max}(S)$,
 								$\Vec{r}(\gnv) \in D$.
\end{corollary}
\section{$\gk$ Properties in the Genric Extension}
				\label[section]{sec:KappaProperties}
\begin{claim} \label[claim]{PrikryProperty}
The forcing notion $\PP$ is of Prikry type.
\end{claim}
\begin{proof}
Assume $p \in \PP$ is a condition and 
					$\gs$ is a formula in the $\PP$-forcing language.\
We will be done by exhibiting a direct extension $p^* \leq^* p$
			 such that $p^* \decides \gs$.
Set $D = \setof {q \leq p}{q \decides \gs}$.
The set $D$ is dense open, hence by \cref{DenseHomogen}
		 there is a direct extension $p^* \leq^* p$,
		a $p^*$-tree $S$,
	and a $\ordered{p^*, S}$-function $r$,
			 such that for each $\gnv \in \Lev_{\max}S$, $\Vec{r}(\gnv) \in D$.
Set $X_0 = \setof {\gnv \in \Lev_{\max}S}{\Vec{r}(\gnv) \forces \lnot \gs}$ and
	$X_1 = \setof {\gnv \in \Lev_{\max}S}{\Vec{r}(\gnv) \forces \gs}$.
Since the sets  $X_0$ and $X_1$ are a disjoint partition of $\Lev_{\max}S$, only one of them is a measure one set.
Fix $i < 2$ such that $X_i$ is a measure one set.
Set $S_i = \setof {\ordered{\gn_0, \dotsc, \gn_k}}
									{\ordered{\gn_0, \dotsc \gn_n}\in X_i,\ k \leq n}$.
Using \cref{GetPreDense}  shrink the trees appearing in the condition $p^*$ so that
	$\setof {\Vec{r}(\gnv)}{\gnv \in \Lev_{\max}S_i}$ is predense below $p^*$.
Thus $p^* \forces \gs_i$, where $\gs_0 = \formula{\lnot\gs}$ and $\gs_1 = \formula{\gs}$.
\end{proof}
\begin{lemma}
$\forces \formula {\gk \text{ is a cardinal}}$.
\end{lemma}
\begin{proof}
If $\mo(\Vec{E}) = 1$  then there are no new bounded subset of $\gk$ in $V^{\PP}$,
	hence no cardinal below $\gk$ is collapsed, hence $\gk$ is preserved.
	If $\mo(\Vec{E}) > 1$ then an unbounded number of cardinals below $\gk$ is preserved,  hence $\gk$ is preserved.
\end{proof}
\begin{claim} \label{BecomesSingular}
If $\mo(\Vec{E}) < \gk$ is regular then
	$\forces \formula{\cf \gk =\cf \mo(\Vec{E})}$.
\end{claim}
\begin{proof}
It is immediate $\forces \formula{\cf\gk \leq \cf\mo(\Vec{E})}$.
Hence we need to show $\forces \formula{\cf\gk \not< \cf\mo(\Vec{E})}$.
Assume $\gs < \gk$ and $p \forces \formula{\gs < \cf\mo(\Vec{E}) \text{ and } 
			\GN{f}\func \gs \to \gk}$.
We will be done by exhibiting  a direct extension $p^* \leq^* p$ such that
		$p^* \forces \formula{\GN{f} \text{ is bounded}}$.
	Let $\ordered{N, f^*}$	 be a good pair such that $p, \GN{f}, \gs \in N$ and $f^* \leq^* f^{p_{n^p-1}}$.
Shrink $T^{p_{n^p-1}}$ so as to satisfy for each
		$\gn \in T^{p_{n^p-1}}$, $\ogn(\gk) > \gs$.

Factor $\PP$ as follows.
Set $P_0 = \setof{s \leq p\restricted n^p-1}{\exists q \leq p_{n^p-1}\ s \append q \in \PP}$
	and
	$P_1 = \setof{q \leq p_{n^p-1}}{\exists s \leq p \restricted n^p-1\  s \append q \in \PP}$.
For each $\gx < \gs$ work as follows.
				Set $D_\gx = \setof {q \leq p_{ n^{p}-1}}
					{\text{There exists a } P_0\text{-name 
					 $\GN{\gr}$ such that } q \forces_{P_1} \formula{\GN{f}(\gx) = \GN{\gr}}}$.
Since $D_\gx \in N$ is a dense open subset of $\PP$ below $p_{n^p-1}$ there is a
	a direct extension  $p^\gx = \ordered{f^*,T^\gx} \leq^* p_{n^p-1}$, 
				a $p^\gx$-tree $S^\gx$, and a $\ordered{p^\gx, S^\gx}$-function $r_\gx$
									satisfying for each $\gnv \in \Lev_{\max}S^\gx$, $\Vec{r}_\gx(\gnv) \in D_\gx$.
		Thus for each $\gnv \in \Lev_{\max} S^\gx$ there is a $P_0$-name
		$\GN{\gr}^{\gx,\gnv}$ so that
			$\Vec{r}_\gx(\gnv)\forces_{P_1} \formula{\GN{f}(\gx) = \GN{\gr}^{\gx,\gnv}}$.
Since $\power{P_0}<\gk$ there is $\gz^{\gx,\gnv} < \gk$ such that
		$p\restricted n^p-1 \forces_{P_0} \formula{\GN{\gr}^{\gx,\gnv}<\gz^{\gx,\gnv}}$.

	Let $m_\gx$ be a function witnessing $S^\gx$ is a $p_{n^p-1}$-tree, i.e.,
		$m_\gx \func \set{\emptyset} \union \Lev_{<\max}S \to \mo(\vE)$ is a function
		satisfying for each $\gnv \in \dom m_\gx$,
				$\Suc_S(\gnv) \in E_{m_\gx(\gnv)}(f^{p_{n^p-1}})$.
    (We use the convention $\Suc_S(\ordered{}) = \Lev_0(S)$.)
 			Since $\mo(\Vec{E}) < \gk$ we can remove a measure zero set from $S^\gx$ (and $\dom m_\gx$)
				and get 		for each
		$\gnv_0, \gnv_1 \in \dom m^\gx$, if $\power{\gnv_0} = \power{\gnv_1}$ then
					$m_\gx(\gnv_0) = m_\gx(\gnv_1)$.
												Thus $\power{\ran m_\gx} < \gw$.
													Set $\gt_\gx = \sup \ran m_\gx$.
Shrink $T^\gx$ so that 
						$\setof {\Vec{r}_\gx(\gnv)}{\gnv \in \Lev_{\max}S^\gx}$ is predense below $p^\gx$.
										Note, if $\gm \in \Lev_0 T^{\gx}$, $\gnv \in \Lev_{\max}S^\gx$,  
														$\mo(\gm) > \gt_\gx$ and $\gnv \not< \gm$, then
																$\Vec{r}(\gnv) \ncompatible p^{\gx}_{\ordered{\gm}}$.
					Hence $p\restricted n^p-1 \append p^{\gx}_{\ordered{\gm}} \forces \formula{\GN{f}(\gx) < \sup \setof {\gz^{\gx,\gnv}}{\gnv\in \Lev_{\max}S^\gx, \gnv < \gm}}$.					

	Set $T^* = \bigintersect_{\gx < \gs} T^\gx$ and
		$p^* = p\restricted n^p-1 \append \ordered{f^*, T^*}$.
						We claim $p^* \forces \formula{\GN{f} \text{ is bounded}}$.
						To show this	 set $\gt = \sup \setof {\gt_\gx}{\gx < \gs}$.
						Note $\gt < \mo(\Vec{E})$.
							Since $\setof{p^*_{\ordered{\gm}}}{\gm \in T^{*},\ 
										\mo(\gm) = \gt}$ is
									predense below $p^*$ it is enough to show that
											$p^*_{\ordered{\gm}} \forces \formula{\GN{f} \text{ is bounded}}$
													for each $\gm \in T^{*}$ such that 
															$\mo(\gm) = \gt$.
									So fix $\gm \in T^{*}$ such that 
															$\mo(\gm) = \gt$.
									Set $\gz = \sup \setof {\gz^{\gx,\gnv}}{\gx < \gs, \gnv\in \Lev_{\max} S^\gx, \gnv<\gm}$.
										Note $\gz < \gk$.
					We get for each $\gx < \gs$,
					$p^*_{\ordered{\gm}} \leq^* p\restricted n^p-1 \append p^{\gx}_{\ordered{\gm}} \forces \formula {\GN{f}(\gx) < \sup \setof {\gz^{\gx,\gnv}}{\gnv\in \Lev_{\max}S^\gx, \gnv < \gm}<\gz<\gk}$.
\end{proof}
\begin{claim} [Gitik]
If $\mo(\Vec{E}) \in [\gk, \gl)$ and $\cf(\mo(\vE)) \geq \gk$
	then 	$\forces \formula{\cf\gk = \gw}$.
\end{claim}
\begin{proof}
Fix a condition $p \in \PP$ such that 
		$\mo(\vE)+1 \subseteq  \in \dom f^{p_{n^p-1}}$.
Partition $T^{p_{n^p-1}}$  into $\mo(\vE)$ 
	disjoint subsets 
		$\setof {A_\gx}{\gx < \mo(\vE)}$
	by setting
	 for each $\gx < \mo(\vE)$,
\begin{align*}
	A_{\gx} = \setof {\gn \in T^{p_{n^p-1}}}{\gx \in \dom \gn,\ 
			\mo(\gn(\gk)) = \otp((\dom \gn) \intersect \gx)}.
\end{align*}
Let $G$ be generic.
Choose a condition $p \in G$.
Let $\ordof {\gn_\gx}{\gx<\gk}$ be the increasing enumeration of the set 
		$\setof {\gn_0, \dotsc, \gn_k}{p\restricted n^p-1 \append
						 p_{n^p-1\ordered{\gn_0, \dotsc, \gn_{\gk}}} \in G}$.
Set $\gz_0 = 0$.
For each $n < \gw$ set
	$\gz_{n+1} = \min \setof {\gx > \gz_n}{
					\gn_\gx \in A_{sup ((\dom \gn_{\gz_n}) \intersect \mo(\vE)) }}$	.
We are done since $\gk = \sup_{n<\gw} \gz_n$.
\end{proof}
Using the same method as above we get the following claim.
\begin{claim}
If $\mo(\Vec{E}) \in [\gk, \gl)$ and $\cf(\mo(\vE)) < \gk$
	then 	$\forces \formula{\cf\gk = \cf(\mo(\vE))}$.
\end{claim}
\begin{claim} \label{BecomesRegular}
If $\mo(\Vec{E}) = \gl$ then
	$\forces \formula{\gk \text{ is regular}}$.
\end{claim}
\begin{proof}
Assume $\gs < \gk$ and $p \forces \formula{\GN{f}\func \gs \to \gk}$.
We will be done by exhibiting  a direct extension $p^* \leq^* p$ such that
		$p^* \forces \formula{\GN{f} \text{ is bounded}}$.
	Let $\ordered{N, f^*}$	 be a good pair such that $p, \GN{f}, \gs \in N$ and $f^* \leq^* f^{p_{n^p-1}}$.
Shrink $T^{p_{n^p-1}}$ so as to satisfy for each
		$\gn \in \Lev_0(T^{p_{n^p-1}})$, $\ogn(\gk) > \gs$.

Factor $\PP(\vE)$ as follows.
Set $P_0 = \setof{s \leq p\restricted n^p-1}{\exists q \leq p_{n^p-1}\ s \append q \in \PP(\vE)}$
	and
	$P_1 = \setof{q \leq p_{n^p-1}}{\exists s \leq p\ \restricted n^p-1 s \append q \in \PP(\vE)}$.
For each $\gx < \gs$ work as follows.
				Set $D_\gx = \setof {q \leq p_{ n^{p}-1}}
					{\text{There exists a } P_0\text{-name 
					 $\GN{\gr}$ such that } q \forces_{P_1} \formula{\GN{f}(\gx) = \GN{\gr}}}$.
Since $D_\gx \in N$ is a dense open subset of $\PP$ below $p_{n^p-1}$ there is a
	a direct extension  $p^\gx = \ordered{f^*,T^\gx} \leq^* p_{n^p-1}$, 
				a $p^\gx$-tree $S^\gx$, and a $\ordered{p^\gx, S^\gx}$-function $r_\gx$
									satisfying for each $\gnv \in \Lev_{\max}S^\gx$, $\Vec{r}_\gx(\gnv) \in D_\gx$.
		Thus for each $\gnv \in \Lev_{\max} S^\gx$ there is a $P_0$-name
		$\GN{\gr}^{\gx,\gnv}$ so that
			$\Vec{r}_\gx(\gnv)\forces \formula{\GN{f}(\gx) = \GN{\gr}^{\gx,\gnv}}$.
Since $\power{P_0}<\gk$ there is $\gz^{\gx,\gnv} < \gk$ such that
		$p\restricted n^p-1 \forces_{P_0} \formula{\GN{\gr}^{\gx,\gnv}<\gz^{\gx,\gnv}}$.

Let $m_\gx$ be a function witnessing $S^\gx$ is a $p_{n^p-1}$-tree, i.e.,
		$m_\gx \func \set{\emptyset} \union \Lev_{<\max}S \to \mo(\vE)$ is a function
		satisfying for each $\gnv \in \dom m_\gx$,
				$\Suc_S(\gnv) \in E_{m_\gx(\gnv)}(f^{p_{n^p-1}})$.
    (We use the convention $\Suc_S(\ordered{}) = \Lev_0(S)$.)
Since $\gl=\mo(\Vec{E})$ is regular and $\power{S^\gx} < \gl$ we get
		$\gt_\gx = \sup \ran m_\gx < \gl$.
Shrink $T^\gx$ so that $\setof {\Vec{r}_\gx(\gnv)}{\gnv \in \Lev_{\max}S^\gx}$
			 is predense below $p^\gx$.
										Note, if $\gm \in T^{\gx}$, $\gnv \in \Lev_{\max}S^\gx$,  
														$\mo(\gm) > \ogm(\gt_\gx)$ and $\gnv \not< \gm$, then
																$\Vec{r}(\gnv) \ncompatible p^{\gx}_{\ordered{\gm}}$.
					Hence $p\restricted n^p-1 \append p^{\gx}_{\ordered{\gm}} \forces \formula{\GN{f}(\gx) < \sup \setof {\gz^{\gx,\gnv}}{\gnv\in \Lev_{\max}S^\gx, \gnv < \gm}}$.					

	Set $T^* = \bigintersect_{\gx < \gs} T^\gx$ and
		$p^* = p\restricted n^p-1 \append \ordered{f^*, T^*}$.
						We claim $p^* \forces \formula{\GN{f} \text{ is bounded}}$.
						To show this	 set $\gt = \sup \setof {\gt_\gx}{\gx < \gs}$.
						Note $\gt < \mo(\Vec{E})=\gl$.
							Since $\setof{p^*_{\ordered{\gm}}}{\gm \in T^{*},\ 
										\mo(\gm) = \ogm(\gt)}$ is
									predense below $p^*$ it is enough to show that
											$p^*_{\ordered{\gm}} \forces \formula{\GN{f} \text{ is bounded}}$
													for each $\gm \in \Lev_0 T^{*}$ such that 
															$\mo(\gm) = \ogm(\gt)$.
									So fix $\gm \in \Lev_0 T^{*}$ such that 
															$\mo(\gm) = \ogm(\gt)$.
									Set $\gz = \sup \setof {\gz^{\gx,\gnv}}{\gx < \gs, \gnv\in \Lev_{\max}S^\gx, \gnv<\gm}$.
										Note $\gz < \gk$.
					We get for each $\gx < \gs$,
					$p^*_{\ordered{\gm}} \leq^* p\restricted n^p-1 \append p^{\gx}_{\ordered{\gm}} \forces \formula {\GN{f}(\gx) < \sup \setof {\gz^{\gx,\gnv}}{\gnv\in \Lev_{\max} S^\gx, \gnv < \gm}<\gz<\gk}$.
\end{proof}
\begin{definition}
An ordinal $\gr < \mo(\Vec{E})$ is a repeat point of $\Vec{E}$
	if for each $d \in [\gee]^{<\gl}$,
			$\bigintersect _{\gx < \gr}E_\gx(d) = \bigintersect_{\gx < \mo(\Vec{E})} E_\gx(d)$.
\end{definition}
\begin{lemma}
Assume $\gr < \mo(\Vec{E})$ is a repeat point of $\Vec{E}$.
\begin{enumerate}
\item
	If $p,q \in \PP$ are compatible then
			$j_{E_\gr}(p)_{\ordered{\mc_\gr(p)}}$ and $j_{E_\gr}(q)_{\ordered{\mc_\gr(q)}}$ are compatible.
\item
	For each $p \in \PP$ there is a direct extension $p^* \leq^* p$ such that
		$j_{E_\gr}(p^*)_{\ordered{\mc_\gr(p^*)}} \decides \formula{ \mc_\gr(p^*) \in j_{E_\gr}(\GN{A})}$.
\end{enumerate}
\end{lemma}
\begin{proof}
\begin{enumerate}
\item
	Let $r \leq p,q$.
	By definition of the order there are extensions  $p' \leq p$ and $q' \leq q$ such that
		$r \leq^* p',q'$.
	By elementarity $j_{E_\gr}(p')_{\ordered{\mc_\gr(p')}} \leq j_{E_\gr}(p)_{\ordered{\mc_\gr(p)}}$ and
				$j_{E_\gr}(q')_{\ordered{\mc_\gr(q')}} \leq j_{E_\gr}(q)_{\ordered{\mc_\gr(q)}}$.
	Thus we will be done by showing
		$j_{E_\gr}(p')_{\ordered{\mc_\gr(p')}}$ and
		$j_{E_\gr}(q')_{\ordered{\mc_\gr(q')}}$ are compatible.
		
	So, without loss of generality assume $p$ and $q$ are $\leq^*$ compatible.
	By elementarity 
		$j_{E_\gr}(p)$ and	$j_{E_\gr}(q)$ are compatible.
	Note
	\begin{align*}
		& p_{n^p-1} = j_{E_\gr}(p_{n^p-1})_{
						\ordered{\mc_\gr(p_{n^p-1})}\downarrow}
						\intertext{and}
		& q_{n^q-1} = j_{E_\gr}(q_{n^q-1})_{
					\ordered{\mc_\gr(q_{n^q-1})}\downarrow}.
		\end{align*}
		Then
	\begin{align*}
		& j_{E_\gr}(p)_{\ordered{\mc_\gr(p_{n^p-1})}} = 
				p \append \ordered{ 
						j_{E_\gr}(p_{n^p-1})_{\ordered{\mc_\gr(p_{n^p-1})}\uparrow}} 
	\intertext{and}
		& j_{E_\gr}(q)_{\ordered{\mc_\gr(q_{n^q-1})}} =
			 q \append \ordered{ 
			 	j_{E_\gr}(q_{n^q-1})_{\ordered{\mc_\gr(q_{n^q-1})}\uparrow}}.
	\end{align*}
	We are done.
\item
	Let $\ordered{N, f^*}$ be a good pair such that $f^* \leq^* f^{p_{n^p-1}}$ and $p, \GN{A} \in N$.
		Set $T = \gp^{-1}_{f^*, f^{p_{n^p-1}}}T^p$.
	Fix $\gn \in T$ and consider the condition $p \restricted n^p-1 \append \ordered{f^*,T}_{\ordered{\gn}}$.
  By the Prikry property there is 
			$s \leq^* p \restricted n^p-1$,
		 				$r \leq^* \ordered{f^*, T}_{\ordered{\gn}\downarrow}$,
		 					and $q \leq^* \ordered{f^*, T}_{\ordered{\gn}\uparrow}$ such that
				$s \append r \append q \decides \formula{\gn \in \GN{A}}$.
Since the set $\setof{t \leq p}{t \decides \formula{\gn \in \GN{A}}}$ is dense open
		below $p$ and belongs to $N$, we get by \cref{BasicReflection} that there is a set
		$T_1(\gn) \in \Vec{E}(f^*)$ such that
$s \append r \append \ordered{f^*_{\ordered{\gn}\uparrow}, T_1(\gn)} \decides \formula{\gn \in \GN{A}}$.

Thus for each $\gn \in T$ there is
			$s(\gn) \leq^* p \restricted n^p-1$,
		 				$r(\gn) \leq^* \ordered{f^*, T}_{\ordered{\gn}\downarrow}$,
							and $T_1(\gn) \in \Vec{E}(f^*)$ such that
			$s(\gn) \append r(\gn) \append \ordered{f^*_{\ordered{\gn}\uparrow}, T_1(\gn)} \decides \formula{\gn \in \GN{A}}$.
We can find a set $T_{=\gr} \in E_\gr(f^*)$ and
	$s \leq^* p \restricted n^p-1$ such that
				 for each $\gn \in T_{=\gr}$, $s(\gn) = s$.
Thus for each $\gn \in T_{=\gr}$,
	$s \append r(\gn) \append \ordered{f^*_{\ordered{\gn}\uparrow}, T_1(\gn)} \decides \formula{\gn \in \GN{A}}$.
Then by removing a measure set from $T_{=\gr}$ we can have either
\begin{align*}
	& 	\forall \gn \in T_{=\gr}\ 
							s \append r(\gn) \append \ordered{f^*_{\ordered{\gn}\uparrow}, T_1(\gn)} \forces \formula{\gn \in \GN{A}}
	\intertext{or}
& 	\forall \gn \in T_{=\gr}\ 
				 s \append r(\gn) \append \ordered{f^*_{\ordered{\gn}\uparrow}, T_1(\gn)} \forces \formula{\gn \notin \GN{A}}.
\end{align*}
	Let $g = f^* \union f^{j(r)(\mc_\gr(f^*))}$.
	Set $T^* = \gp^{-1}_{g,f^*}T^{j_{E_\gr(r)}(\mc_\gr(f^*))} \intersect \gp^{-1}_{g,f^*}\dintersect_{\gn \in T_{=\gr}} T_1(\gn)$.
Setting $p^* = s \append \ordered{g,T^*}$ we get 
	for each $\gn \in T^*_{=\gr}$,
		$p^*_{\ordered{\gn}} \leq^*
			s \append r(\gn) \append \ordered{f^*_{\ordered{\gn}\uparrow}, T_1(\gn)}$.
Thus by removing a measure zero set from $T^*_{=\gr}$ we get either
\begin{align*}
	& 	\forall \gn \in T^*_{=\gr}\ 
							p^*_{\ordered{\gn}} \forces \formula{\gn \in \GN{A}}
	\intertext{or}
& 	\forall \gn \in T^*_{=\gr}\ 
							p^*_{\ordered{\gn}} \forces \formula{\gn \notin \GN{A}}.
\end{align*}
Going to the ultrapower we get
			$j_{E_\gr}(p^*)_{\ordered{\mc_\gr(g)}} \decides \formula{\mc_\gr(g) \in j_{E_\gr}(\GN{A})}$.
\end{enumerate}
\end{proof}
\begin{corollary}
Assume $\gr < \mo(\Vec{E})$ is a  repeat point of $\Vec{E}$.
Then $ \forces \formula {\gk \text{ is measurable}}$.
\end{corollary}
\begin{proof}
If $G \subseteq \PP$ is generic then it is a simple matter to check that
\begin{align*}
U = \setof {\GN{A}[G]}
			{p \in G,\ 
			j_{E_\gr}(p)_{\ordered{\mc_\gr(p_{n^p-1})}} \forces \formula{\mc_{\gr}(p_{n^p-1})\in j_{E_\gr}(\GN{A})}}
\end{align*}
is the witnessing ultrafilter.
\end{proof}
\begin{claim} \label{BecomesMeasurable}
If $\mo(\Vec{E}) = \gl^{++}$ then
	$\forces \formula {\gk \text{ is measurable}}$.
\end{claim}
\begin{proof}
By the previous corollary it is enough to exhibit $\gr < \gl^{++}$ 
	which is a repeat 	point of $\Vec{E}$.
Fix  $d \in [\gee]^{<\gl}$ 
	and consider the sequence 
			$\ordof {\bigintersect_{\gx' < \gx} E_{\gx'}(d)}{\gx < \gl^{++}}$.
This is a $\subseteq$-decreasing sequence of filters on $\OB(d)$.
Since there are $\gl^+$ filters on $\OB(d)$ there is
	$\gr_d < \gl^{++}$ such that
		$\bigintersect_{\gx < \gr_d} E_{\gx}(d) = 
			\bigintersect_{\gx < \gl^{++}} E_{\gx}(d)$.
Set $\gr = \sup \setof{\gr_d}{d \in [\gee]^{<\gl}}$.
Then $\gr$ is a repeat point of $\vE$.
\end{proof}
%
%
%
%
%
%
%
%
%
%

%
%
%
%
%
%
\end{document}